\documentclass{article}

\usepackage[latin1]{inputenc}
\usepackage{verbatim,tabmac_avi}
\usepackage{airport_diagrams_new}

\usepackage{amsthm,amsmath,times,amsopn,amsfonts,amssymb,epsfig,anysize}
\usepackage{graphicx}
\usepackage[all]{xy}
\usepackage{pifont,txfonts,graphicx,geometry}

\marginsize{2.9cm}{2.9cm}{2.5cm}{3.2cm}
\usepackage{latexsym,graphics,color}
\usepackage{tabmac_avi}

\newtheorem{theorem}{Theorem}

\newtheorem{prop}[theorem]{Proposition}
\newtheorem{cor}[theorem]{Corollary}
\newtheorem{lemma}[theorem]{Lemma}

\newtheorem{example}[theorem]{Example}

\newtheorem{corollary}[theorem]{Corollary}
\newtheorem{definition}[theorem]{Definition}

\theoremstyle{remark}
\newtheorem{remark}[theorem]{Remark}

\def\charge{ {\rm {charge}}}
\def\weight{ {\rm {weight}}}
\def\cocharge{ {\rm {cocharge}}}
\def\shape{ {\rm {shape}}}

\def \SSYT{\mathcal{SSYT}}
\def \SSCT{\mathcal{SSCT}}

\newcommand{\wt}{\mathrm{wt}}
\newcommand{\word}{\mathrm{word}}

\renewcommand{\th}{^{\text{th}}}

\DeclareMathOperator{\inv}{inv}

\def\weight{ {\rm {weight}}}
\def\maj{ {\rm {maj}}}

\def\zigmaj { {\rm {zmaj}}}
\def\ski{ {\rm {betrayal}}}
\def\shape{ {\rm {shape}}}

\def \SSYT{\mathcal{SSYT}}

\def \F{\mathcal{F}}
\def \CT{\mathcal{CT}} % these are colored tabloid
\def \C{\mathcal{C}} % these are circloid
\def \T{\mathcal{T}} %tabloid
\def \RPP{\mathrm{RPP}}
\def\sh{ {\rm {sh}}}

\def\des{ {\rm {Des}}}

\def\jhat{ {\hat \jmath}}

\newcommand\Ccolor[1]{ C_{#1}}
\newcommand\Cletter[1]{ C^{\geq {#1}}}

\begin{document}

\begin{center}
{\Large Colorful combinatorics and Macdonald polynomials 
%and applications to equivariant Schubert calculus
}

\bigskip
\bigskip

{\large 
Ryan Kaliszewski 
 and Jennifer Morse
\footnote{Partially supported by the NSF grant DMS-1600953.}
}

\bigskip

{\it Lehigh University

Department of Mathematics

Bethlehem, PA 19104 }

\bigskip

{\it University of Virginia

Department of Mathematics

Charlottesville, VA 22903} 
\end{center}

\begin{abstract}
The non-negative integer cocharge statistic on words was introduced in the 1970's by Lascoux and Sch\"utzenberger 
to combinatorially characterize the Hall-Littlewood polynomials.  Cocharge has since been used to
explain phenomena ranging from the graded decomposition of Garsia-Procesi modules  to
the cohomology structure of the Grassman variety.
Although its application to contemporary variations of these problems had been deemed
intractable, we prove that the two-parameter, symmetric Macdonald polynomials are 
generating functions of a distinguished family of {\it colored} words.
Cocharge adorns one parameter and the second measure its deviation from cocharge on words without color.
We use the same framework to expand the plactic monoid, apply Kashiwara's crystal theory
to various Garsia-Haiman modules, and to address problems in $K$-theoretic Schubert calculus.
%Further, the same framework expands the plactic monoid, pairs with Kashiwara's crystal theory,
%and applies to problems in $K$-theoretic Schubert calculus.
\end{abstract}

\begin{center}
{\it Dedicated to Alain Lascoux}
\end{center}

%\tableofcontents
\section{Introduction}

Kostka-Foulkes polynomials, $K_{\lambda\mu}(t)\in\mathbb N[t]$, 
describe the connections between characters of $GL_n(\mathbb{F}_q)$ and
the Hall-Steinitz algebra~\cite{Green}, give characters of cohomology rings of Springer fibers 
for ${\rm GL}_n$~\cite{Spr,HoSpr}, and are graded multiplicities of modules for the general linear group 
obtained by twisting functions on the nullcone by a line bundle~\cite{Bry}.
Lusztig~\cite{[Lu],Lus} showed they are
the $t$-analog of the weight multiplicities in the irreducible representations of the classical Lie algebras,
$$
K_{\lambda\mu}(t)=\sum_{\sigma \in W}(-1)^{\ell(\sigma)}\mathcal P_t(\sigma(\lambda+\rho))-(\mu+\rho))\,,
$$
obtained from a $t$-deformation of Kostant's partition function $\mathcal P$ defined by
$$
\prod_{\text{positive roots }\alpha} \frac{1}{(1-tx^\alpha)} = \sum_\beta \mathcal P_t(\beta) x^\beta\,.
$$

%The cohomology rings of Spring fibers can be realized as a quotient $R/I$ in the ring of polynomials 
%which decomposes, as an $S_n$-module, into irreducible components.
%and the decomposition 
%, as an $S_n$-module, 
%De Concini and Procesi~\cite{DeP} realized the cohomology rings of Springer fibers as a quotient in
%the ring of polynomials.  Garsia and Procesi studied the quotient as an $S_n$-module and found
%the Hall-Littlewood 

Algebraically, Kostka-Foulkes polynomials are the entries in transition matrices between the 
Schur and the Hall-Littlewood $\{H_\mu(x;t)\}$ bases for the algebra $\Lambda$ of symmetric functions 
in variables $x = x_1,x_2,\ldots,$ over the field $\mathbb Q(t)$.  In fact, this reflects the graded decomposition 
of a simple quotient of the coinvariant ring viewed as an $S_n$-module~\cite{GP};
each irreducible submodule of polynomials with homogeneous degree $r$ 
corresponds to a Schur function with coefficient $t^r$, and the sum over
all irreducibles corresponds to a Hall-Littlewood polynomial.

Kostka-Foulkes polynomials are wrapped in the most fundamental combinatorial ideas.
Namely, the set of words forms a monoid under the operations of RSK-insertion and jeu-de-taquin.
The monoid structure was first motivated by Sch\"utzenberger~\cite{Schutz} in his proof that the Schubert 
structure constants for the cohomology of the Grassmann variety are enumerated by Young tableaux with a 
distinguished {\it Yamanouchi} property.
The structure is compatible with the assignment of each word to a non-negative integer (statistic) 
called {\it cocharge}.  Lascoux and Sch\"utzenberger~\cite{[LS2]} proved that 
the generating function for tableaux weighted by this statistic is precisely $K_{\lambda\mu}(t)$.
Consequently, the spectrum of topics surrounding Kostka-Foulkes polynomials is accessible 
from a purely combinatorial study of cocharge.  For example,
\begin{equation}
\label{kftableauxintro}
\tilde
H_\mu({x};t)=\sum t^{\cocharge(T)}\,s_{{\rm shape}(T)}({x})\,,
\end{equation}
summing over all Young tableaux $T$ with $\mu_1$ ones, $\mu_2$ twos, etc.

In the 1980's, Macdonald introduced a basis for $\Lambda$ over the field $\mathbb Q(q,t)$
to unify the Hall-Littlewood and the Jack polynomials (wave equations 
of the Calogero-Sutherland-Moser model~\cite{For}).
Ensuant studies of the basis have impacted an impressive range 
of areas including the representation theory of quantum groups \cite{EK},
double affine Hecke algebras \cite{Cherednik},
%the Calogero-Sutherland model in particle physics \cite{CSmodel},
the shuffle algebra and diagonal harmonics~\cite{HHLRU},
the geometry of Hilbert schemes~\cite{Haiman},
affine Schubert calculus \cite{[LLM]},
the elliptic Hall algebra of Shiffmann-Vasserot \cite{ScVa},
and extensions of HOMFLY polynomials for knot invariants \cite{ors}.

Early characterizations of Macdonald polynomials were oblique, revealing little more 
than that they are elements of $\mathbb Q(q,t)[x_1,x_2,\ldots]$.  Nevertheless, using 
brute force to compute examples, Macdonald conjectured that the entries $K_{\lambda\mu}(q,t)$ 
of certain of their transition matrices lie in $\mathbb N[q,t]$.  Garsia modified Macdonald's 
polynomials so that $K_{\lambda\mu}(q,t)$ appeared as Schur expansion coefficients of the
resulting polynomials $\tilde H_\mu(x;q,t)$, thus
piquing the interest of representation theorists for whom Schur functions
are synonymous with irreducible $S_n$-modules.
The {\it $q,t$-Kostka coefficients} in
\begin{equation}
\label{qtintro}
\tilde H_\mu(x;q,t)=\sum_\lambda K_{\lambda\mu}(q,t)\,s_\lambda(x)\,,
\end{equation}
have since been a matter of great interest.

Rich theories were born from the compelling feature that
the $q,t$-Kostka coefficients reduce to the Kostka-Foulkes polynomials at
$q=0$.
In~\cite{GHqt}, Garsia and Haiman introduced $S_n$-modules $R_\mu$, 
for $\mu$ a partition of $n$, given by the space of polynomials
in variables $x_1, \ldots x_n;y_1, \ldots y_n$ spanned by
all derivatives of a certain simple determinant  $\Delta_\mu$.
They conjectured that the dimension of $R_\mu$ equals $n!$, and 
that the modules provide a representation theoretic framework 
for~\eqref{qtintro}.  Their interpretation was designed 
to imply the Macdonald positivity conjecture.  Haiman spent years putting
together algebraic geometric tools which ultimately led him to prove the
conjectures in~\cite{Haiman}.

Formula~\eqref{kftableauxintro} set the gold standard for defining Macdonald polynomials, but
cocharge was abandoned after efforts to give a manifestly positive formula 
for generic $\tilde H_\mu(x;q,t)$ led no further than the most basic examples.
In 2004, an explicit formula for Macdonald polynomials 
was established by Haglund-Haiman-Loehr.  Rather than using Young tableaux and cocharge, the
formula involves the major index and an intricate 
inversion-like statistic:
\begin{equation}
\label{jimformula}
\tilde H_\lambda({x};q,t)=\sum_F q^{\inv(F)} t^{\maj(F)}
\prod_{i} x_{F(i)}\,,
\end{equation}
over all $\mathbb Z_+$-valued functions (fillings) $F$ on the partition 
$\lambda$.  The Schur expansion was expected to come shortly behind this breakthrough,
but it took another decade even to recover the Hall-Littlewood case.  
In~\cite{Roberts}, Austin Roberts converted the $q=0$ case of~\eqref{jimformula} 
into a new Schur expansion formula:
\begin{equation}
\label{roberts}
\tilde H_\lambda({x};0,t)
=\sum_{F\in \mathcal U\atop \inv(F)=0} t^{\maj(F)}\,s_{\weight(F)}({x})\,,
\end{equation}
over a mysterious subset $\mathcal U$ of fillings (see~\S~\ref{sec:faithful}).
Roberts' questioning of the comparison of his formula with the earlier formulation~\eqref{kftableauxintro}
sparked our interest and led us to revive the study of cocharge.

We discovered that the classical combinatorics of cocharge supports Macdonald polynomials as
naturally as it does the less intricate setting surrounding Hall-Littlewood polynomials.
The key idea is a broadening of the plactic monoid~\cite{LSplactic,LLTplactic} whereby each letter in a word is colored.
Of particular importance is the subset of {\it tabloids}, words with an increasing condition
used by Young to define (Specht) modules.  We prove that Macdonald polynomials are colored tabloid generating functions,
weighted by $\cocharge$ and a {\it betrayal} statistic which measures the variation of $\cocharge$ on colored words 
from its value on usual words.

\noindent
{\bf Theorem.}
{\it
For any partition $\mu$,
\begin{equation}
\label{rjformula}
\tilde H_\mu({x};q,t)=\sum_{T}
 q^{\ski(T)} t^{\cocharge(T)} {x}^{\shape(T)}\,, 
\end{equation}
over colored tabloids with $\mu_1$ ones, $\mu_2$ twos, and so forth.
}

%Colored tabloids, and the associated cocharge underlying Macdonald polynomials,
%have applications which extend those of the tableaux cocharge statistic.

Further applications of colored words are geometrically inspired.
The classical example in Schubert calculus addresses the cohomology of the Grasmann variety
where the structure constants $c_{\lambda\mu}^\nu$ count Yamanouchi tableaux.  Schubert calculus 
vastly expanded with efforts to characterize the structure of $K$-theory and (quantum) cohomology of other varieties;
the problems are a combinatorial search for alternative, or more refined notions, of Yamanouchi.  
Thus, the combinatorial ideas surrounding the plactic monoid are often revisited in Schubert calculus.  
In fact, $c_{\lambda\mu}^\nu$ can be viewed as the number of skew tableaux with zero cocharge and
the broader scope of colored words fits in well.

We extend Van Leeuwen's approach~\cite{VanL} to the Yamanouchi condition
using Young {\it tableaux companions}. We show that colored tabloids serve as {companions} for the generic
$\mathbb Z_+$-valued functions used in the Macdonald polynomials~\eqref{jimformula}.
From this point of view, a {\it super-Yamanouchi} condition arises and is applicable to $K$-theoretic 
Schubert calculus problems as well as Kostka-Foulkes polynomials.
The companion map $\mathfrak c$ simultaneously gives relations between
\begin{itemize}
\item the formulas~\eqref{kftableauxintro} and~\eqref{roberts} for $q=0$ Macdonald polynomials,
\item genomic tableaux of Pechenik-Yong~\cite{PY} and set-valued tableaux of~\cite{Buch}, introduced to 
study $K$-theoretic problems in Schubert calculus, and
\item cocharge and  the Lenart-Schilling statistic~\cite{LenSch} for computing the (negative of the) energy 
function on affine crystals.
\end{itemize}
Colored words also support equivariant $K$-theory of Grassmannians and Lagrangians,
but details are sequestered in a forthcoming paper.

We investigate representation theoretic lines with the theory of crystal bases,
introduced by Kashiwara \cite{Kas1,Kas2} in an investigation of quantized enveloping algebras $U_q(\mathfrak g)$
associated to a symmetrizable Kac--Moody Lie algebra $\mathfrak g$.
Integrable modules for quantum groups
play a central role in two-dimensional solvable lattice models.
When the absolute temperature is zero ($q=0$), there is a distinguished
{\it crystal basis} with many striking features.  
The most remarkable is that the internal structure of an integrable 
representation can be combinatorially realized by associating the 
basis to a colored oriented graph whose arrows are imposed by the 
Kashiwara (modified root) operators.
From the crystal graph, characters can be computed by enumerating elements
with a given weight, and the tensor product decomposition
into irreducible submodules is encoded by
the disjoint union of connected components.  Hence, progress in
the field comes from having a natural realization of crystal graphs.

A double crystal structure on colored tabloids using only
the type-$A$ crystal operators and jeu-de-taquin 
provides a lens giving clarity to problems in Macdonald theory
and in Schubert calculus.
Several crystal graphs arise simultaneously through different colored tabloid manifestations of tabloids.
From these, we deduce Schur expansion formulas for dual Grothendieck polynomials and
the $q=0,1$ cases of Macdonald polynomials.
The $q=1$ result perfectly mimics the classical formula~\eqref{kftableauxintro} for $q=0$.  In particular,

\noindent
{\bf Theorem.}
{\it
For any partition $\mu$,
$$
\tilde H_\mu({x};1,t)= 
\sum
t^{\cocharge(T)}\,s_{\shape(T)} 
\,,
$$
over colored tabloids with column increasing entries.
}

%Another crystal restricts to tabloids with highest weights identified by column-strict tableaux and the LS formula
%\eqref{kftableauxintro} is recovered.  In particular, the crystal graph supports the Garsia-Procesi modules.

\section{Preliminaries}
\subsection{Garsia-Haiman modules}

The algebra $\Lambda$ of symmetric functions in infinitely many indeterminants 
$x_1,x_2,\ldots,$ over the field $\mathbb Q(q,t)$ has bases are indexed by {\it partitions},
$\lambda=(\lambda_1\geq \ldots\geq \lambda_\ell>0)\in\mathbb Z^\ell$.
The {\it monomial basis} is defined by elements $m_\lambda=\sum_{\alpha} x^{\alpha}$ taken over
all distinct rearrangements $\alpha$ of $(\lambda_1,\ldots,\lambda_\ell,0,\ldots)$.
The {\it homogeneous basis} has elements $h_\lambda=h_{\lambda_1}\cdots h_{\lambda_\ell}$,
where $h_r({x})=\sum_{i_1\leq\ldots\leq i_r}x_{i_1}\cdots x_{i_r}$, and 
the {\it power basis} has elements $p_\lambda=p_{\lambda_1}\cdots p_{\lambda_\ell}$, with
$p_r(x)=\sum_i x_i^r$.  A basis element indexed by partition $\lambda$ of {\it degree} $d=\sum_i \lambda_i$ 
(denoted by $\lambda\vdash d$) is a sum of monomials of degree $d$.

%$e_r(x)=\sum_{i_1<\ldots <i_r}x_{i_1} \cdots x_{i_r}$, 
The {\it Hall-inner product}, $\langle\,,\rangle$, is defined on $\Lambda$ by
$$
\langle p_\lambda, p_\mu\rangle = \delta_{\lambda\mu}\,z_\lambda \quad\text{where }
\;\;
z_\lambda = 
\prod_i\alpha_i! i^{\alpha_i}\;\text{ for } \lambda = (\cdots,2^{\alpha_2},1^{\alpha_1})\,,
$$
and $\delta_{\lambda\mu}$ evaluates to 0 when $\lambda\neq\mu$ and is otherwise 1.
In fact, the basis of {\it Schur functions}, $s_\lambda$, can be defined as the 
unique orthonormal basis which is unitriangularly related to the monomial basis;
for each $\lambda\vdash d$,
\[ 
s_\lambda = m_\lambda + \sum_{\mu\vdash d\atop\mu\lhd\lambda} a_{\lambda\mu}\,m_\mu\,,\]
where {\it dominance order} $\mu\lhd \lambda$
is defined by $\lambda_1+\cdots+\lambda_k\leq \mu_1+\cdots+\mu_k$ for all $k$.

Macdonald~\cite{Macbook} proved the existence of another basis of polynomials,
$P_\lambda(x;q,t)$, also unitriangularly related to the monomials, but orthogonal
with respect to the $q,t$-deformation of $\langle\,,\rangle$ ,
$$
\langle p_\lambda,p_\mu\rangle_{q,t} = \delta_{\lambda\mu}\,
z_\lambda\,\prod_j\frac{1-q^{\lambda_j}}{1-t^{\lambda_j}}\,.
$$
Of interest to combinatorialists, but not apparent from the definition, 
he conjectured that $P_\lambda(x;q,t)$ have certain transition coefficients lying in $\mathbb Z_{\geq 0}[q,t]$.
Garsia modified $P_\lambda(x;q,t)$ into 
polynomials $\tilde H_\lambda(x;q,t)$ to rephrase Macdonald's conjecture as one about Schur positivity:  
the {\it $q,t$-Kostka coefficients} in
\begin{equation}
\label{qtkostka}
\tilde H_\mu(x;q,t)=\sum_\lambda K_{\lambda\mu}(q,t)\,s_\lambda(x)
\end{equation}
lie in $\mathbb Z_{\geq 0}[q,t]$.  
 See~\eqref{jimformula} for a precise definition of $\tilde H_\mu(x;q,t)$.

Garsia's approach appealed to a broader audience. Namely, results of Frobenius dictate
that a positive sum of Schur functions models the decomposition of an $S_n$-representation into 
its irreducible submodules.
Namely, for $\sigma\in S_n$ and $\lambda\vdash n$, the value of the irreducible character $\chi^\lambda$ of $S_n$
at $\sigma$ arises in
\begin{equation}
\label{the:frob}
s_\lambda = 
\frac{1}{n!}
\sum_{\sigma\in S_n}
\chi^\lambda(\sigma) \, p_{\tau(\sigma)}\,,
\end{equation}
where $\tau(\sigma)$ is the cycle-type of $\sigma$.
Define the linear {\it Frobenius map}
from class functions on $S_n$ to symmetric functions of degree $n$ by
$$
F_\chi = 
\frac{1}{n!}
\sum_{\sigma\in S_n}
\chi^\lambda \, p_{\tau(\sigma)}\,,
$$
and consider the Frobenius image of a doubly-graded $S_n$-module $\mathcal M=\bigoplus_{r,s} M_{r,s}$,
$$
F_{{\rm char}(\mathcal M)}(x;q,t) = 
\sum_{r,s} t^r\,q^s\,F_{{\rm char}(M_{r,s})}
\,.
$$
The function $F_{{\rm char}(\mathcal M)}(x;q,t)$ is thus a positive sum of
Schur functions with coefficients in $\mathbb Z_{\geq 0}[q,t]$ by~\eqref{the:frob}.

So launched the search for a bi-graded module $\mathcal M$ for which $\tilde H_\lambda(x;q,t)$ is the
Frobenius image.  Garsia and Procesi settled the $q=0$ case and gave the perfect guide~\cite{GP}.  
In particular, they gave an algebraic approach to Hotta and Springer's result that $K_{\lambda\mu}(0,t)$ describes
the multiplicities 
of $S_n$ characters $\chi^\lambda$ in the graded character of the cohomology 
ring of a Springer fiber, $B_\mu$.  The cohomology ring $H^*(B_\mu)$ 
can be defined by a particular quotient,
$$
R_\mu(y)=\mathbb C[y_1,\ldots,y_n]/I_\mu\,,
$$
of the coinvariant ring $R_{1^n}(y)=\mathbb C[y_1,\ldots,y_n]/\langle e_1,\ldots,e_n\rangle$.
$R_\mu(y)$ is the {\it Garsia-Procesi module} under the natural $S_n$-action permuting variables;
they proved the ideal $I_\mu$ is generated by {\it Tanisaki generators}, defined to be
the elementary symmetric functions $e_k(S)$
in the variables $S=\{y_{i_1},\ldots,y_{i_r}\}\subset \{y_1,\ldots,y_n\}$
when $r>k>|S|-\text{\# cells of $\mu$ weakly east of column $r$}$.

The simplicity of Garsia and Procesi's definition
led them to an algebraic proof that $K_{\lambda\mu}(0,t)\in\mathbb N[t]$
and offered an attack on the $q,t$-Kostka polynomials.  Given that the Frobenius 
image of $R_\mu(y)$ is $\tilde H_\mu(x;0,t)$, the task was to define an $S_n$-module
$$
R_\mu(x;y)= \mathbb C[x_1,\ldots,x_n;y_1,\ldots,y_n]/J_\mu\,,
$$
under the the diagonal $S_n$-action, simultaneously permuting the $x$ and $y$ variables,
so that 
\begin{equation}
\label{HisFrob}
\tilde H_\mu(x;q,t) = F_{{\rm char}(R_\mu(x;y))}(x;q,t)\,.
\end{equation}
Garsia and Haiman found just the candidate;  it is the ideal

$$
J_\mu = \left\{f: f\left(\frac{\partial}{\partial x_1}, \cdots, \frac{\partial}{\partial x_n}, 
\frac{\partial}{\partial y_1}, \cdots, \frac{\partial}{\partial y_n}\right) \Delta_\mu=0
\right\}
\,,
$$
where $\Delta_\mu$ is a generalization of the Vandermonde defined using a 
graphical depiction of $\mu$.
A {\it lattice square} $(i,j)$ lies in the $i$th row and $j$th column of
$\mathbb N\times \mathbb N$.  
The (Ferrers) {\it shape} of a {\it composition} $\alpha=(\alpha_1,\ldots,\alpha_\ell)
\in\mathbb Z_{\geq 0}^\ell$ is the subset of $\mathbb N\times\mathbb N$
made up of $\alpha_i$ lattice squares left-justified in the $i^{th}$ row, 
for $1\leq i\leq \ell$.  
A lattice square inside a shape $\alpha$ is called a {\it cell}.
Given $\mu\vdash n$, the cells $\{(r_1,c_1),\ldots,(r_n,c_n)\}$ in $\mu$
define 
$$
\Delta_\mu=
{\rm det}
\begin{bmatrix}
x_1^{c_1}y_1^{r_1}& x_2^{c_1}y_2^{r_1}&\cdots & x_n^{c_1}y_n^{r_1}\\
\vdots & & &\vdots\\
x_1^{c_n}y_1^{r_n}& x_2^{c_n}y_2^{r_n}&\cdots & x_n^{c_n}y_n^{r_n}\\
\end{bmatrix}
\,.
$$
Although the construction of the modules $R_\mu(x;y)$ is quite simple, 
the proof of~\eqref{HisFrob} required sophisticated geometric techniques 
developed by Haiman~\cite{Haiman}.  

\subsection{Cocharge}
How $R_\mu(x;y)$ decomposes into irreducible submodules remains an open problem.  It is 
particularly intriguing in light of the perfect description for decomposing $R_\mu(y)$
in terms of the following statistic on words.  Given a word $w$ in the alphabet $\mathcal A$, 
$w_{\mathcal B}$ is the subword of $w$ restricted to letters of $\mathcal B\subset \mathcal A$. 
When $\mathcal B=\{i\}$, we use simply $w_i=w_{\mathcal B}$.
The {\it weight} of a word $w$ is the composition $\alpha$, where
$\alpha_i$ is the number of times $i$ appears in $w$.
A word with weight $(1,\ldots,1)$ is called {\it standard}.
The {\it cocharge} of a standard word $w\in S_n$ is defined by writing $w$ counter-clockwise 
on a circle with a $\star$ between $w_1$ and $w_n$, attaching a label to each letter, 
and summing these labels.
The labels are determined iteratively starting by labeling 
1 with a zero.  Letter $i$ is then given the same label as $i-1$ as
long as $\star$ lies between $i-1$ and $i$ (reading clockwise) 
and it is otherwise incremented by 1.

The cocharge of a word $w$ with weight $\mu\vdash n$ is defined by
writing $w$ counter-clockwise on a circle and computing the cocharge of
$\mu_1$ standard subwords of $w$.  Letters
of the $i\th$ standard subword are adorned with a subscript $i$
and this subword is determined iteratively 
from $i=1$ as follows: clockwise from $\star$, choose the first occurrence 
of letter 1 and proceed on to the first occurrence of letter 2.  Continue 
in this manner until $\mu_1'$ has be given the index $1$.  
Start again at $\star$ with $i+1$, repeating the process
on letters without a subscript.
The cocharge of $w$ is the sum of the cocharge of each standard subword.
The {\it charge} of a word $w$ of weight $\mu$ is
$$
\charge(w)=n(\mu)-\cocharge(w)\,,
$$
where $n(\mu)=\sum_i (i-1)\mu_i$.

\begin{example}
The words $w_1=6714235$ and $w_2=64534223511123$ written counter-clockwise on circles
$$
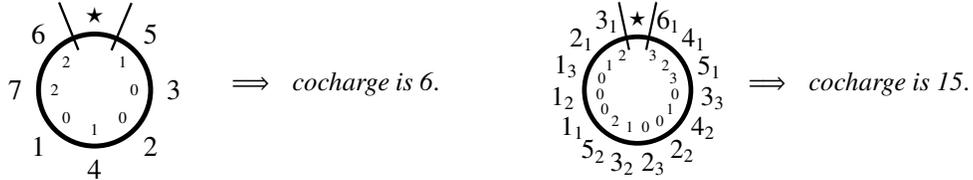
\doubleairport{5,3,2,4,1,7,6}{1,0,0,1,0,2,2}{7}{2.5cm}\quad
\implies \;\text{cocharge is 6}.
\qquad \qquad
\doubleairport{6_1,4_1,5_1,3_3,4_2,2_2,2_3,3_2,5_2,1_1,1_2,1_3,2_1,3_1}
{3,2,3,0,1,0,0,1,2,0,0,0,1,2}
{14}{2.5cm}
\implies \;\text{cocharge is 15}.
$$
\end{example}

Kostka-Foulkes polynomials require only words coming from Young tableaux.
Use $\alpha\models n$ to denote that $\alpha$ is a composition of {\it degree} $n=|\alpha|=\alpha_1+\alpha_2+\cdots$.
For compositions $\alpha$ and $\beta$ where $\alpha_i\leq \beta_i$ for all $i$,
we say $\alpha\subseteq\beta$.  The {\it skew shape} of $\alpha\subseteq\beta$ 
is $\beta/\alpha$, defined by the set theoretic difference of their cells and of 
degree $|\beta|-|\alpha|$.
%The ring of coinvariants is isomorphic to the left regular representation of $S_n$, for which
%Young introduced the combinatorics of tableaux.
A {\it (semi-standard) tableau} is the filling of a skew shape with positive integers 
which increase up columns and are not decreasing 
along rows (from west to east).  

\begin{definition}
The {\it reading order} of any collection $S\subset\mathbb N\times\mathbb N$
is the total ordering on elements in $S$ defined by saying that lattice squares 
decrease from left to right, starting in the highest 
row and moving downward.
\end{definition}

Given a tableau $T$, the {\it reading word} $w=\word(T)$ is defined by taking $w_i$ to be the letter in the
$i\th$ cell of $T$, where cells are read in decreasing reading order.  The {\it weight} of tableau $T$ 
is the weight of its reading word and $T$ is called {\it standard} when $\word(T)$ is standard.
For skew shape $\lambda/\mu$ of degree $n$ and $\gamma\models n$, the set of tableaux of 
shape $\lambda/\mu$ and weight $\gamma$ is denoted by $\SSYT(\lambda/\mu,\gamma)$.
Lascoux and Sch\"utzenberger~\cite{[LS2]} proved, for partitions $\lambda$ and $\mu$ of the same degree,
\begin{equation}
\label{kftableaux}
K_{\lambda\mu}(0,t)=
\sum_{T\in \SSYT(\lambda,\mu)} t^{\cocharge(T)}\,,
\end{equation}
where the cocharge of a tableau $T$ is defined by $\cocharge(\word(T))$.

A similarly beautiful formula for the $q,t$-Kostka polynomials has been
actively pursued for decades.  Because $K_{\lambda\mu}(1,1)=|\SSYT(\lambda,1^n)|$,
the endgame is to establish a 
formula for $K_{\lambda\mu}(q,t)$ by attaching a $q$ and a $t$ 
weight to each standard tableau.

\subsection{Macdonald polynomials}

Although the Schur expansion of Macdonald polynomial still eludes us, Jim Haglund made a breakthrough in 2004 
by proposing a combinatorial formula for $\tilde H_\mu(x;q,t)$.
Rather than using semi-standard tableaux and cocharge, 
different statistics are associated to arbitrary {fillings}.

A {\it filling} $F$ of shape $\beta/\alpha$ and weight $\gamma\models|\beta|-|\alpha|$ is
any placement of letters from a word with weight $\gamma$ into
shape $\beta/\alpha$.  The entry in row $r$ and column $c$ of $F$ is 
denoted by $F_{(r,c)}$, and the set of fillings of shape $\beta/\alpha$ and weight  $\gamma$ 
is $\F(\beta/\alpha,\gamma)$.  Immediate from the definition is 
\begin{equation}
\label{multi}
|\F(\alpha/\beta,\gamma)|=\binom{|\gamma|}{\gamma_1,\gamma_2,\ldots,\gamma_{\ell(\gamma)}}\;.
\end{equation}

For a filling $F$ of partition shape $\lambda$, an {\it inversion triple} is 
a triple of entries $(r,t,s)$ which are arranged in a collection of cells in $F$ 
of the form
$$
\tableau[scY]{ r&\ldots & t \cr s \cr}\,,
$$
and meeting the criteria that $r=s\neq t$ or some cycle of $(r,t,s)$ 
is decreasing, i.e. $r>s>t$, $s>t>r$, or $t>r>s$.  If the cells containing $r$ and $t$ are in the first row we envision that $s=0$.
The {\it inversion statistic} is the number $\inv(F)$ of inversion triples in $F$. 
The {\it major index} of $F$ is
$$
\maj(F)=
\sum_{F_{(r,c)}>F_{(r-1,c)}}(\lambda_c'-r+1)
%=
%\sum_{x\in \Des(F)} (leg(x)+1)\,,
$$
%where the leg of a cell $x$ is the number of entries above it
%and 
where 
%a cell $(r,c)$ has a descent if its entry is larger than the entry in cell $(r-1,c)$ and 
$\lambda'_c$ is the number of cells in column $c$ of $\lambda$.  
Every partition $\lambda$ has a {\it conjugate} $\lambda'$ given by reflecting shape $\lambda$ about $y=x$.  
Alternatively, a filling $F$ has a \emph{descent} at cell $(r,c)$ when $F_{(r,c)}>F{(r-1,c)}$
and $\maj(F)$ is the
number of descents of $F$, each weighted by the number of cells appearing weakly above it in $F$.

\begin{example}
The following filling $F\in\F((333),(342))$ has $\maj(F)=1+2+1=4$ and $\inv(F)=4$.
$$
\text{\rm Descents:}\; 
{\tableau[scY]{ \tf 3, 2,\tf 3 \cr 2,\tf 2,1 \cr 2,1,1
 }}
\qquad
\qquad
\text{\rm Inversion triples:}\;
\tableau[scY]{ 3,\tf 2,\tf 3 \cr 2,\tf 2,1 \cr 2,1,1 }
\quad
\tableau[scY]{ 3,2,3 \cr \tf 2,2,\tf 1 \cr \tf 2,1,1 }
\quad
\tableau[scY]{ 3,2,3 \cr 2,2,1 \cr \tf 2, \tf 1,1}
\quad
\tableau[scY]{ 3,2,3 \cr 2,2,1 \cr \tf 2,1,\tf 1}
$$
\end{example}

\begin{theorem}~\cite{HHL}
\label{Hinm}
For any partition $\lambda$,
\begin{equation}
\label{jimformulabody}
\tilde H_\lambda({x};q,t)=\sum_{F\in \F(\lambda,\cdot)}
q^{\inv(F)} t^{\maj(F)}\,x^{\weight(F)}\,.
\end{equation}
\end{theorem}

\section{Frobenius image of Garsia-Haiman modules using cocharge}

We introduce a new combinatorial structure and prove that the
Macdonald polynomials are generating functions attached to cocharge
and a second statistic called betrayal.
%The combinatorial structure is naturally compatible with a crystal structure 
%leading immediately to the Lascoux-Sch\"utzenberger formula when $q$ is set 
%to 0 and to Theorem~\eqref{t1qtkostka} when $q=1$.  
%\label{kftableaux}

\subsection{Colored words and circloids}

A {\it colored letter} $x_i$ is a letter $x$ in an alphabet $\mathcal A$ 
adorned with a subscript (its {\it color}) $i$ from $\mathcal A$.
A {\it colored word} $w$ is a string of distinct colored letters.
The {\it weight} of $w$ is a skew composition recording the colors which adorn each letter;
$\weight(w)=\alpha/\beta$ where $\{\beta_x+1,\ldots,\alpha_x\}$
are the colors attached to letter $x$ in $w$.  When $\beta\models 0$,
we simply say the weight of $w$ is $\alpha$.

%\jennifer{Check this next paragraph carefully, there were some things I changed e.g. order of $\gamma$ etc}
Colored words also come equipped with shapes which are assigned using the {\it prismatic order} on 
colored letters: $u_v>r_c$ when $r<u \text{ or } (r=u \text{ and } c>v)$.  Equivalently,
\[
\text{ $u_v>r_c  \iff \text{ cell } (u,v)$ precedes $(r,c)$ in reading order.}
\]
A strict composition (one without zero entries) $\gamma\models n$ is a {\it shape} admitted by a 
colored word $w=w_n\cdots w_1$ if $w_1>\cdots>w_{\gamma_1}$, $w_{\gamma_1+1}>\cdots>w_{\gamma_1+\gamma_2}$, 
and so forth.  A weak composition $\gamma'\models n$ is a shape admitted by $w$ when the
strict composition obtained by removing the zeroes from $\gamma'$ is a shape admitted by $w$.

The set of all shapes admitted by a colored word $w=w_1\cdots w_n$ corresponds multisets that contain 
$$
{\rm Des}(w) = \{ p : w_p<w_{p+1}\}\,,
$$
under the bijection sending compositions of degree $n$ to sub-multisets of $\{1,\ldots,n-1\}$
defined by
$$
(\gamma_\ell,\ldots,\gamma_1)\mapsto
{\rm set}(\gamma)=\{\gamma_1,\gamma_1+\gamma_2,\ldots,\gamma_1+\cdots+\gamma_{\ell(\gamma)-1}\}
\,.  A shape is a strict composition if and only if it corresponds to a true set.
$$
%The {\it refinement order} on strict compositions of $n$ is defined by
%$\alpha\succeq\beta$ when ${\rm set}(\alpha)\subseteq{\rm set}(\beta)$.

\begin{example}
The colored word $w=3_21_33_11_12_22_11_22_33_3$ has weight $(3,3,3)$ and 
$\des(w)=\{1,3,6\}$.  It thus admits shape $\lambda=(3,3,2,1)$ and, for example,
$\beta=(3,1,2,2,1)$.
\end{example}
%\begin{example}
%\label{ex:comp}
%For $\alpha=(1,2,3,3)$ and $\beta=(1,2,2,1,3)$, 
%$\alpha\succeq\beta$ since
%${\rm set}(\alpha)=\{1,3,6\}$ and ${\rm set}(\beta)=\{1,3,5,6\}$.
%\end{example}

A circular representation of colored words is convenient when attaching statistics.
We write a colored word $w$ counter-clockwise on a circle and separate its letters 
into sectors to give a concept of shape.

%%% Circloid definition %%%
\begin{definition}
A circloid $C$ of shape $\gamma\models n$ is a placement of $n$ distinct colored letters on 
the perimeter of a subdivided circle such that, reading clockwise from a distinguished point $\star$,
$\gamma_x$ colored letters lie in decreasing prismatic order in sector $x$, for $x=1,\ldots,\ell(\gamma)$.
\end{definition}

Each circloid $C$ is uniquely associated to a colored word $w$ of the same shape by reading the
letters of $C$ in counter-clockwise order.  The {\it weight} of $C$ is defined to be $\weight(w)$.
The set of circloids of weight $\alpha/\beta$ and shape $\gamma$ is denoted by $\C(\gamma,\alpha/\beta)$.
Note that letter $b$ appears in a circloid $C\in\C(\cdot,\alpha/\beta)$ exactly $\alpha_b-\beta_b$ 
times since there are $\alpha_b-\beta_b$ colors needed to adorn the set of $b$'s.

\begin{example}
Circloids $C_1\in\C((3,3,2,1),(3,3,3))$ and $C_2\in\C((3,1,2,2,1),(3,3,3))$
with underlying $w=3_21_33_11_12_22_11_22_33_3$ are
\[
C_1= \airport{3_3,2_3,1_2,2_1,2_2,1_1,3_1,1_3,3_2}{3,3,2,1}{2.2cm}\;\;\qquad
C_2= \airport{3_3,2_3,1_2,2_1,2_2,1_1,3_1,1_3,3_2}{3,1,2,2,1}{2.2cm}\;\;
\]
\end{example}

If unspecified, entries and positions of a circloid are always taken clockwise.
For example, $3_2$ and $3_3$ lie between $1_3$ to $2_3$  in $C_1$
since $3_2$ and $3_3$ are passed when reading clockwise from $1_3$ to $2_3$. 
We consider the following two restrictions of a circloid $C$, 
$$
\Ccolor{i}= \{x_y\in C: y=i\}
\qquad
\text{and}
\qquad
\Cletter{j}= \{x_y\in C: x\geq j\}\,.
$$

For the traditionalists, we interpret circloids as fillings of shapes.
For compositions $\gamma$ and $\alpha/\beta$, a {\it colored tabloid} 
$T\in\CT(\gamma,\alpha/\beta)$ 
is a filling of shape $\gamma$ with colored letters so that row
entries are increasing from west to east under the prismatic order
and the colors adorning letter $x$ in $T$ are $\{\beta_x+1,\ldots,\alpha_x\}$.
As expected, $\alpha/\beta$ is called the {\it weight} of $T$.
It is straightforward to see that
a bijection is given by the map 
$$
\iota : \C(\gamma,\alpha/\beta)\longrightarrow \CT(\gamma,\alpha/\beta)\,,
$$
defined by putting the colored letters of sector $r$ from circloid $C$
into row $r$ of shape $\gamma$ so that the row is colored increasing
from west to east, for each $r=1,\ldots,\ell(\gamma)$.

\subsection{Circloid statistics}

The cocharge statistic on words naturally extends to circloids.
Macdonald polynomials turn out to be generating functions
of circloids, weighted by cocharge and a second statistic which measures the variation 
of cocharge from the Lascoux-Sch\"utzenberger statistic.

For a circloid $C\in\C(\cdot,\mu)$ of partition weight $\mu$,
the {\it cocharge} of $C$ is defined by
$$
\cocharge(C)=\sum_{i=1}^{\ell(\mu)} \cocharge(\Ccolor{i})\,.
$$
\begin{remark}
\label{re:uncoloredcharge}
Any word $w$ with partition weight $\mu$ can be uniquely
identified with a circloid $C\in\C(1^n,\mu)$ of the same cocharge.
%and zero betrayal.
Place the letters of $w$ counter-clockwise on a circle and 
color according to the labeling for standard subwords. That is,
moving clockwise from $\star$, label the first 1 with $i=1$.  By iteration, 
the first $x+1$ encountered in the clockwise reading from $x_1$ is colored 1.  
Once $\mu_1'$ letters have been labeled by $1$, repeat with $2$ on 
uncolored letters, and so on.
\end{remark}

The second statistic measures how different a coloring is 
from standard subwords. When choosing which letter $x$ to color $j$, each
candidate passed over in clockwise order increases the statistic by 1.
Precisely, 
$$
\ski(C)=\sum_{i\geq 1}\sum_j s_{i,j}\,,
$$
where $s_{i,j}$ is the number of $i_{\jhat}$ with $\jhat>j$ 
lying between $(i-1)_j$ and $i_j$ in $C$,
with the understanding that $0_j=\star$ for all $j=1,\ldots,\mu_1$.

\begin{example}
The circloids  in the previous example have a betrayal of 2 and cocharge of 4.
\end{example}

It is through the lens of circloids that we can prove cocharge is as fundamental to the $q,t$-Macdonald 
setting as it is to Kostka-Foulkes and Hall-Littlewoods polynomials.  We show that 
a Macdonald polynomial is none other than the shape generating function of circloids weighted by
cocharge and betrayal.  Moreover, the result follows straightforwardly 
from a correspondence between circloids and skew fillings.

%\footnote{This map is related to one used 
%in~\cite{km} to establish a set of representatives for (equivariant) $K$-theory classes}.

\begin{definition}
\label{fillingsmap}
The map $\mathfrak f$ acts on a circloid $C$ of shape $\gamma$ by placing entry $x$ 
in cell $(r,c)$, for each colored letter $r_c$ in sector $x$, moving through
sectors $x=1,\ldots,\ell(\gamma)$.
\end{definition}
 
\begin{example}
The action of $\mathfrak f$ on two circloids:
\[
\airport{2_3,1_4,3_1, 2_1, 2_2, 3_2}{2,3,1}{2.2cm}\;\;
\quad \mapsto \quad 
\;\tableau[scY]{2, 3 |2,  2,1 | \bl,\bl,\bl, 1 }
\qquad \qquad
\qquad \qquad
\airport{ 3_3,2_3,1_2,2_1,2_2,1_1,3_1,1_3,3_2}{3,3,2,1}{2.2cm}\;\;
\quad \mapsto \quad 
\;\tableau[scY]{3, 4,1 |2,  2,1 |2,1,3}
\;\;. \]
\end{example}

\begin{theorem}
\label{the:rjformula}
For any partition $\lambda$,
\begin{equation}
\label{rjformula}
\tilde H_\lambda({x};q,t)=\sum_{C\in \C(\cdot,\lambda)}
 q^{\ski(C)} t^{\cocharge(C)} {x}^{\shape(C)}\,.
\end{equation}
\end{theorem}

\begin{proof}
Consider any compositions $\gamma$ and $\beta\subseteq\alpha$ where $|\gamma|=|\alpha|-|\beta|$.
We first establish that  $\mathfrak f$ is a bijection where
\begin{equation}
\label{codomainf}
\mathfrak f : \C(\gamma,\alpha/\beta)\longrightarrow \F(\alpha/\beta,\gamma)\,.
\end{equation}
Given a circloid $C$, let $f=\mathfrak f(C)$.  
Each letter $r$ in sector $x$ of $C$ corresponds to an entry $x$ in row $r$ of $f$ 
implying that $C\in \C(\gamma,\alpha/\beta)$ if and only if $f\in \mathcal F(\alpha/\beta,\gamma)$.

Note similarly that any other circloid $D$ where $\mathfrak f(C)=\mathfrak f(D)$ 
lies in $\C(\gamma,\alpha/\beta)$.  Consider the set of colored letters
$\{r^{(1)}_{c^{(1)}},\ldots,r^{(\gamma_x)}_{c^{(\gamma_x)}}\}$ in sector $x$ of $C$.  
By definition of $\mathfrak f$, $f_{(r^{(i)},c^{(i)})}=x$ for $i=1,\ldots,\gamma_x$.
Therefore, $\{r^{(1)}_{c^{(1)}},\ldots,r^{(\gamma_x)}_{c^{(\gamma_x)}}\}$ is also
the set of colored letters in sector $x$ of $D$.  Since colored letters lie in 
unique decreasing prismatic order within sectors, every sector of
$C$ and $D$ is the same and we see that $C=D$.
That $\mathfrak f$ is bijective then follows by noting that
the number of fillings given in~\eqref{multi} matches the number of
circloids $C\in\C(\gamma,\cdot)$.  That is,
again viewing letters from sector $x$ of $C$ 
as a subset $\{r^{(1)}_{c^{(1)}},\ldots,r^{(\gamma_x)}_{c^{(\gamma_x)}}\}$ 
of the distinct $|\gamma|$ colored letters in $C$, we see that
\begin{equation}
|\C(\gamma,\alpha/\beta)|=
\binom{|\gamma|}{\gamma_1,\gamma_2,\ldots,\gamma_{\ell(\gamma)}}
\,.
\end{equation}

We next restrict our attention to circloid $C\in\C(\gamma,\lambda)$ for $\lambda\vdash n$
and claim that
\begin{equation}
\label{inv2skip}
\cocharge(C)=\maj(f) \qquad\text{and}\qquad \ski(C)=\inv(f)\,,
\end{equation}
for $f=\mathfrak f(C)$.
Since $\maj$ is computed on columns, we need only verify that $\cocharge(\Ccolor{i})$ 
equals the maj of column $i$ in $f$ to prove that $\maj(f)=\cocharge(C)$.
A slight reinterpretation of the cocharge definition gives $\cocharge(C_i)=\sum_r L_r$ 
where $L_r=\lambda_i'-r+1$ when $r_{i}$ occurs (clockwise) between $(r-1)_i$ 
and $\star$ in $C$ and otherwise $L_r=0$.
In fact, since $r_i>(r-1)_i$ and sectors are prismatic order decreasing, 
$L_r\neq 0$ if and only if $r_i$ lies in a sector $y$ strictly larger than the sector 
$x$ containing $(r-1)_i$.  
On the other hand, the action of $\mathfrak f$ dictates that
$r_i$ is in sector $y$ and $(r-1)_i$ is in sector $x$ of $C$ precisely when 
$y$ lies in cell $(r,i)$ above $x$ in  cell $(r-1,i)$ of $f$.

We next claim that $\ski(C)=\inv(f)$ using the observation that $\ski(C)=\sum_{i\geq 1}
\left(|I^i_1|+\cdots +|I^i_{\lambda_i'}|\right)$, where
$$
I^i_{j}=\{\,\jhat>j: i_{\jhat} \text{ lies between $(i-1)_{j}$ and $i_j$\}}\,,
$$
with the convention that $0_j=\star$ for all $j$.  
For any $j$, $1_j$ is in sector $x$ of $C$ if and only if entry $x$ lies in $(1,j)$ of $f$.
Note that $\jhat\in I_j^1$ implies $i_\jhat$ lies in sector $y<x$ since $i_j<i_{\jhat}$ and
therefore
$\jhat\in I_j^1$ corresponds uniquely to an inversion of $x$ with the entry 
$y$ in $(i,\jhat)$ of $f$.  When $i>1$,
for each pair of $(i-1)_j$ in sector $x$ and $i_j$ in sector $y$ of $C$, 
one of the following relations concerning the sector $z$ with $\jhat\in I_j^i$ 
must be true: $x=y$ and $z\neq x$, $x < z <y$, 
$y < x < z$, or $z < y <x$.  
Correspondingly, entries $x, y,$ and $z$ in cells $(i-1,j),(i,j),$ and $(i,\jhat)$ of $f$ respectively,
form a triple inversion.
\end{proof}

We can extend the definitions of cocharge and betrayal to colored tabloid,
\[ \cocharge(T)=\cocharge(\iota^{-1}T) \qquad\text{and}\qquad
\ski(T)=\ski(\iota^{-1}T). \]
Immediately following from Theorem~\ref{the:rjformula} is an expression using 
charge
and one using fixed weight colored tabloid.

\begin{cor}
\label{cor:tildeH}
For any partition $\mu$,
\[
\tilde H_\mu(X;q,1/t)t^{n(\mu)} = \sum_{C\in \C(\cdot,\mu)}
q^{\ski(C)}\, t^{\charge(C)}\, x^{\shape(C)}
=
\sum_{T\in\CT(\cdot,\mu)}
q^{\ski(T)}\,t^{\charge(T)}\,x^{\shape(T)}\;.\]
\end{cor}

\section{Colorful companions}

%Circloids, although arising to reflect Macdonald symmetric functions, 
%provide a framework which brings together combinatorial ideas ranging 
%from Schubert calculus to Kashiwara's theory of crystal bases.

The subset of Young tableaux with an additional {\it Yamanouchi} condition is of particular
importance; its cardinality gives tensor product multiplicities
of ${\rm GL}_n$, the Schur expansion coefficients in a product of Schur functions, 
and the Schubert structure constants in the cohomology 
of the Grassmannian ${\rm Gr}(k,n)$ of $k$-dimensional subspaces of $\mathbb C^n$.

For partition $\lambda$, $T_\lambda$ denotes the unique tableau of shape and 
weight $\lambda$.  A word $w$ is {\it $\lambda$-Yamanouchi} when 
$w\cdot \word(T_\lambda)= b_n\cdots b_2 b_1$ has the property that the weight 
of each suffix $b_j\cdots b_2b_1$ is a partition.  
A filling is {\it $\lambda$-Yamanouchi } when
its reading word is $\lambda$-Yamanouchi and a circloid is 
{\it $\lambda$-Yamanouchi} when the counter-clockwise reading of its 
letters is $\lambda$-Yamanouchi.  
A $\emptyset$-Yamanouchi object is simply called {\it Yamanouchi}.

\begin{remark}
\label{re:chargeyama}
Since a word of weight $\mu$ has zero charge only when
every standard subword is the maximal length permutation,
zero charge matches the Yamanouchi condition.
\end{remark}

Because many open problems in representation theory, geometry, and symmetric function theory
involve a search for contemporary notions of Yamanouchi and tableaux to characterize
mysterious invariants, the Yamanouchi condition has been revisited often from different viewpoints.  
The combinatorics of circloids naturally captures several of these simultaneously.

\subsection{Companions and the Yamanouchi condition}

Van Leeuwen addresses the classical Littlewood-Richardson rule by rephrasing the Yamanouchi 
condition on skew tableaux $P$ in terms of {\it companion tableaux}.  
A companion of $P$ is any skew tableau $Q$ such that the entries in row $x$ 
match the row positions of letters $x\in P$ and are aligned to meet the 
condition that entries increase up columns.
He proves that a Yamanouchi tableau $P$ always has a companion tableau of (straight) partition shape $\mu$.  

We forsake the column increasing condition and instead view a companion as 
the {tabloid} where rows are uniquely aligned into a straight shape.
Such a companion of semi-standard tableau $P$ is precisely the tabloid obtained by 
ignoring colors of $\iota\circ\mathfrak f^{-1}(P)$.  This approach opens the door
to a more inclusive study allowing for companions of arbitrary fillings.

\begin{definition}
\label{def:companion}
The {\it companion map} is the bijection,
$$
\mathfrak c=\iota\circ \mathfrak f^{-1} : 
\F(\nu/\lambda,\mu)\longrightarrow \CT(\mu,\nu/\lambda)\,.
$$
The companion of a filling $F\in \F(\nu/\lambda,\mu)$ is the unique colored tabloid $\mathfrak c(F)$.
\end{definition}

Following directly from the definition of $\mathfrak f$ 
and $\iota$, the action of $\mathfrak c$ on a filling $F$ takes entry $e$ in 
cell $(r,c)$ to the colored letter $r_c$ placed in row $e$ 
of $T$, arranged so that each row of $T$ is colored increasing.
Companions give a valuable mechanism to study Yamanouchi related problems.

\begin{definition}
A filling $F$ is super-Yamanouchi when 
the non-decreasing rearrangement of entries within each row 
is a Yamanouchi tabloid.
\end{definition}

\begin{prop}
\label{prop:yamacs}
Given partitions $\mu$ and $\nu/\lambda$, consider a filling
$F\in \F(\nu/\lambda,\mu)$ and its companion 
$T=\mathfrak c(F)$.
\begin{enumerate}
\item $F$ is Yamanouchi if and only if entries of $T$ are prismatic increasing in columns,
\item $F$ is a super-Yamanouchi filling if and only if letters of $T$ 
increase in columns,
\item
letters of $F$ increase in columns if and only if $T$ is $\lambda$-Yamanouchi.
\end{enumerate}
\end{prop}
\begin{proof}
(1) $F$ is Yamanouchi if and only if the letter $x$ in a cell $(r,c)$ of 
$F$ can be paired uniquely with an $x-1$ in some cell $(\hat r,\hat c)$ occurring 
after $(r,c)$ in reading order.  Equivalently, each entry 
$r_c$ in row $x$ of $T$ pairs uniquely with an entry 
$\hat r_{\hat c}\leq r_c$ in row $x-1$ (in prismatic order).
Since row entries in a colored tabloid lie in increasing prismatic order, 
such a pairing can occur exactly when the entry immediately below $r_c$ 
is smaller in prismatic order.

(2) 
Consider a colored tabloid $T$ where letters do not strictly increase up some column.
If columns of $T$ are not prismatic increasing, $F$ is not Yamanouchi by (1).
Otherwise, we can choose $b$ to be the rightmost column of $T$ with an $r_c$ in 
row $x$ and an $r_{\hat c}$ in row $x-1$ where $c<\hat c$. 
Correspondingly, $F$ has an $x$ in cell $(r,c)$ and an $x-1$ in $(r,\hat c)$.
Since entries in rows of $T$ are prismatic non-decreasing and $r_c$ and $r_{\hat c}$ 
lie in column $b$, the  subset of cells in $F$ weakly smaller than $(r,\hat c)$ 
in reading order contain $b$ $x-1$'s and $b$ $x$'s.  However, the filling $\hat F$
obtained by rearranging letters in row $r$ of $F$ into weakly increasing order
is not Yamanouchi since the $x$ in column $c<\hat c$ of $F$ moves to 
the east of all $x-1$'s in that row.

On the other hand, a colored tabloid $T$ with letters increasing up columns has
prismatic increasing columns and therefore $F$ is Yamanouchi by (1).  
Suppose
that $F$ has an $x$ and an $x-1$ in cells $(r,c)$ and $(r,\hat c)$, respectively, 
such that when letters in row $r$ are put into weakly increasing order, the resulting
filling is not Yamanouchi.
Since $F$ is Yamanouchi, this can only happen if $\hat c>c$ and there is an equal number of $x-1$'s and $x$'s 
in the subset of cells of $F$ occurring weakly after $(r,\hat c)$ in the reading order of cells.
However, under the $\mathfrak f$-correspondence, $T$ has an $r_c$ in row $x$ and an $r_{\hat c}$ in row $x-1$ lying
in the same column.  The violation of increasing columns establishes the claim.

(3) Given $F\in \F(\nu/\lambda,\mu)$ is column increasing, 
construct the unique filling $\hat F\in\F(\nu,(1^{|\lambda|},\mu))$ by
replacing each $i\in F$ with $i+|\lambda|$ and putting $1,2,\ldots,|\lambda|$ 
into cells of $\lambda$ so the reading word taken from these cells is
$|\lambda|\cdots 21$.
Note that $\hat C=\mathfrak f^{-1}(\hat F)$ differs from
$\mathfrak f^{-1}(F)$ by the deletion of the first $|\lambda|$ sectors.

A letter in an arbitrary cell $(r,c)$  of $\hat F$
is larger than the letter in cell $(r-1,c)$ if and only if entry $r_c$ occurs in a later sector 
then $r-1_c$ of $\hat C=\mathfrak f^{-1}(\hat F)$.
This is equivalent to the Yamanouchi condition
on $\hat C$; $r$ can be paired with a letter $r-1$ which occurs 
earlier than it, for each letter $r$ in $\hat C$.
The claim follows by noting that the counter-clockwise reading of letters from
the first $|\lambda|$ sectors of $\hat C$ is $\word(T_\lambda)$.  
\end{proof}

\subsection{Reverse companions}

The initial study of companions involved only the subset of fillings which are semi-standard
tableaux.
Proposition~\ref{prop:yamacs} pinpoints that dropping the row condition and requiring only that letters 
increase in columns of a filling imposes the $\lambda$-Yamanouchi condition on its companion circloid
(or colored tabloid).  On the other hand, it is also natural to examine the subset of
fillings which are {\it tabloids}, that is, fillings which are
non-decreasing in rows from west to east.  Let $\T(\alpha,\beta)$ be the set of tabloids of 
shape $\alpha$ and weight $\beta$.

A distinguished coloring on circloids comes to light under these conditions.  A circloid is
{\it reverse colored} when the colors adorning letter $x$ 
increase clockwise from $\star$, for each fixed letter $x$. 
A tabloid $T$ is reverse colored if $\iota^{-1}(T)$ is a reverse colored circloid.

\begin{remark}
\label{re:manifest}
Since the reverse coloring uniquely assigns a color to each letter of a tabloid,
reverse colored circloids are a manifestation of tabloids.
\end{remark}

\begin{prop}
\label{prop:revcs}
Given compositions $\gamma$ and $\beta\subseteq\alpha$, the companion
$T$ of a filling $F\in \F(\alpha/\beta,\gamma)$
is reverse colored if and only if $F$ is a tabloid.
\end{prop}

\begin{proof}
By definition of $\mathfrak f$,
a filling $F$ has the property that the letter in cell $(r,c)$
is not smaller than the letter in $(r,c-1)$
if and only if $r_c$ does not occur before
$r_{c-1}$ in $\mathfrak f^{-1}(F)$.
\end{proof}

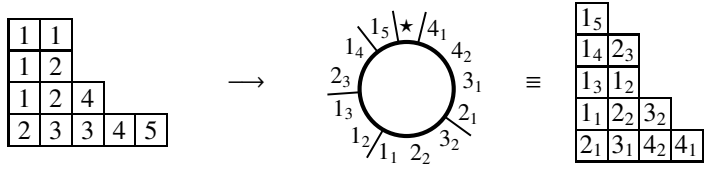
\begin{example}
\[
{ \tableau[scY]{1,1|1,2|1,2,4|2,3,3,4,5} }
\qquad
\longrightarrow
\qquad
\airport{4_1,4_2,3_1,2_1,3_2,2_2,1_1,1_2,1_3,2_3,1_4,1_5}{4,3,2,2,1}{2.2cm}
\quad
\equiv
\quad
{ \tableau[scY]{1_5|1_4,2_3|1_3,1_2|1_1,2_2,3_2|2_1,3_1,4_2,4_1} }
\quad \quad
\]
\end{example}

In particular, the companion $Q$ of a semi-standard tableau filling
$P$ is a reverse colored $\lambda$-Yamanouchi circloid, or equivalently,
is a manifestation of a $\lambda$-Yamanouchi  tabloid
by Remark~\ref{re:manifest}.
Furthermore, when $F$ is both Yamanouchi and a semi-standard tableau, 
we recover the classical result that Yamanouchi tableaux of shape $\nu/\lambda$ 
and weight $\mu$ and $\lambda$-Yamanouchi tableaux 
of shape $\mu$ and weight $\nu-\lambda$ are equinumerous.

\begin{corollary}
\label{cor:revcompanion}
For partitions $\nu/\lambda$ and $\mu$, $P\in\SSYT(\nu/\lambda,\mu)$
if and only if its companion is a $\lambda$-Yamanouchi tabloid, and
$P$ is Yamanouchi if and only if its companion is a 
$\lambda$-Yamanouchi tableau  $Q\in\SSYT(\mu,\nu-\lambda)$.
\end{corollary}

In the combinatorial theory of $K$-theoretic Schubert calculus, 
tableaux are replaced by more intricate combinatorial objects such as
reverse plane partitions, set-valued tableaux, and genomic tableaux.
The later were introduced recently by Pechunik and Yong~\cite{PY}
to solve a difficult problem concerning equivariant $K$-theory of the Grassmannian. 
We have discovered that reverse-colored companions are closely related 
to genomic tableaux and carry out the details separately.
A glimpse of this application is given in \S~\ref{sec:ktheory} where
a crystal structure on reverse colored circloids is used to 
study the representatives for $K$-homology classes of the Grassmannian.

\subsection{Faithful companions}
\label{sec:faithful}

Another useful manifestation of tabloids arises from a second distinguished circloid coloring.
A circloid $C$ is {\it faithfully colored} when, for each $i\geq 1$,
if entries of color $j<i$ are ignored,
the closest $1$ to $\star$ (moving clockwise) 
has color $i=1$ and the closest $x+1$ to $x_i$ has color $i$, for $x\geq 1$.  
A colored tabloid is defined to be faithfully colored if it is the $\iota$-image of
a faithfully  colored circloid.

\begin{example}
A faithfully colored circloid and its corresponding (faithfully colored) tabloid:
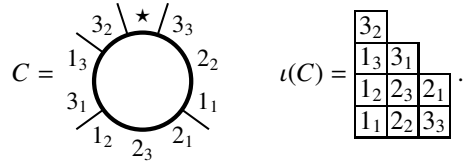
\[ C= \airport{
3_3,2_2,1_1,2_1,2_3,1_2,3_1,1_3,3_2}{3,3,2,1}{2.2cm}
\qquad
\iota(C) = \tableau[scY]{ 3_2 | 1_3,3_1 | 1_2,2_3,2_1 | 1_1,2_2,3_3 }\;. \]
\end{example}

\begin{prop}
\label{prop:charge2maj}
For composition $\alpha$ and partition $\lambda$ where $|\alpha|=|\lambda|$,
the image of the companion map $\mathfrak c$ on 
$$
\mathcal F^*(\lambda,\alpha)=
\{F\in \F(\lambda,\alpha): \inv(F)=0\}
$$ is
the subset of faithfully colored tabloid in $\CT(\alpha,\lambda)$.
\end{prop}
\begin{proof}
The number of inversion triples in a filling $F$ matches the betrayal of $\mathfrak f(F)$ by~\eqref{inv2skip}.
It thus suffices to note, by definition, that restricting the set of circloids to those with zero betrayal gives the subset of
faithfully colored elements.
\end{proof}

%\begin{corollary} \label{HLfixedweight} For any partition $\mu$, 
%\[ \tilde H_\mu(\textbf{x};0,t)=\sum_{T\in\T(\cdot,\mu)}t^{\cocharge(T)}\textbf{x}^{\shape(T)}\;.\]
%\end{corollary}

\begin{theorem}
\label{the:Hsuper}
For any partition $\mu$,
$$
\tilde H_\mu(x;0,t)=
\sum_{
F\in \F^*(\mu,\cdot)\atop
F~super-Yamanouchi} t^{\maj(F)}\,s_{\weight(F)} \,.
$$
\end{theorem}
\begin{proof}
Consider an inversionless filling $F\in \F^*(\mu,\cdot)$ which is super-Yamanouchi.
Note that the weight of $F$ must be a partition $\lambda\vdash |\mu|$ since $F$ is Yamanouchi.
Propositions~\ref{prop:yamacs} and~\ref{prop:charge2maj} give that the $\mathfrak c$-image 
of $F$ is a faithfully colored tabloid with letters which increase up columns.  
Since each tabloid has a unique faithful coloring, ignoring colors gives
the bijection between
\begin{equation}
\label{eq:charge2maj}
\{F\in \F(\mu,\lambda): \inv(F)=0\;\text{and}\;F~super-Yamanouchi\}
\leftrightarrow \SSYT(\lambda,\mu)\,.
\end{equation}
The $\maj(F)=\cocharge(\mathfrak f(F))$ by~\eqref{inv2skip},
and any faithfully colored circloid $C\in \C(\lambda,\mu)$
has the same cocharge as that of its manifest tabloid $T$
by Remark~\ref{re:uncoloredcharge}.
The result then follows from~\eqref{kftableaux}.
\end{proof}

The comparison of Theorem~\ref{the:Hsuper} to~\eqref{roberts} suggests that
the set $\mathcal U$ defined by Roberts is related to the super-Yamanouchi 
condition.
Roberts' formula requires inversionless, Yamanouchi fillings
with an additional property imposed upon entries in a {\it pistol} 
configuration, one that is made up of cells in row $r$ lying in 
columns $1,\ldots,c$ and cells of row $r+1$ lying in columns 
$c,\ldots,\mu_{r+1}$, for any fixed $r,c$.  In our language, a filling
is {\it jammed} if its reverse coloring results in a 
pistol containing both $x_y$ and $(x+1)_{y+1}$ for some letter $x$ with
color $y$.    When a filling is not jammed, we say it is {\it jamless}.

\begin{lemma}
The set of inversionless, super-Yamanouchi fillings is the same as 
the set of inversionless, jamless, Yamanouchi fillings.
\end{lemma}
\begin{proof}
Suppose an inversionless, super-Yamanouchi filling $F$ is jammed and for convenience, consider 
its reverse coloring.  Since $F$ is jammed, it has rows $r$ and $r+1$ with a pistol containing
$x_y$ and $(x+1)_{y+1}$.  If $x_i$ does not lie in a lower row than $(x+1)_j$ of a super-Yamanouchi filling,
then $j<i$.  Therefore, $x_y$ must lie in cell $(r,c)$ of $F$ and $(x+1)_{y+1}$ in 
cell $(r+1,\hat c)$, for some $\hat c\geq c$.  Moreover, $x_{y+1}$ lies in row $r$ of $F$.
Consider the minimal color $i$ adorning $x$ in the set of rows higher than row $r$.
Since $x_{i-1},x_{i-2},\ldots,x_{y+1}$ lie west of column $\hat c$ in row $r$, 
and $F$ is inversionless, an $x+1$ lies above each of these $x$'s.
Therefore, $(x+1)_i$ lies in row $r+1$ contradicting that $F$ is
super-Yamanouchi.

On the other hand, suppose that a filling $F$ is jamless, inversionless, and
Yamanouchi but not super-Yamanouchi.  Then there is some row $r$ and letter $x+1$ in $F$ where
the number $y$ of $x+1$'s weakly below row $r$ is greater than the number of $x$'s below row $r$.  
In particular, $y$ is the color adorning the leftmost $x+1$ in row $r$ in the reverse coloring
of $F$.  Further, the $x$ colored $y$ must lie after this $(x+1)_y$ in reading order since
$F$ is Yamanouchi.  However, it is not super-Yamanouchi and therefore $x_y$ lies in row $r$.
Since $F$ is inversionless and has $(x+1)_y$ west of $x_y$ in row $r$, $x$ must lie below $(x+1)_y$.  
When reverse colored, this $x$ has color $z<y$.  In turn, $z<y$ implies $(x+1)_{z+1}$ must lie weakly 
after $(x+1)_y$ (in reading order).  However, the Yamanouchi condition requires that $(x+1)_{z+1}$ 
lies before $x_z$.  Therefore, the pistol based at $x_z$ contains $(x+1)_{z+1}$ 
contradicting that $F$ is not jammed.
\end{proof}

\begin{cor}
Roberts' formula~\eqref{roberts} and the Lascoux-Sch\"utzenberger formula~\eqref{kftableaux} 
for $q=0$ Macdonald polynomials are related by the companion map $\mathfrak c$.
\end{cor}

\section{Crystals}

The quantum enveloping algebra $U_q(\mathfrak{sl}_{n+1})$  is the $\mathcal Q(q)$-algebra generated by
elements $e_i,f_i, t_i, t_i^{-1}$, for $1\leq i\leq n$, subject to certain relations.
For a $U_q(\mathfrak{sl}_{n+1})$-module $\mathcal M$ and $\lambda\in\mathbb Z^{n+1}$,
the weight vectors (of weight $\lambda$) are elements of the set
$M_\lambda = \{u\in \mathcal M: t_i u = q^{\lambda_i-\lambda_{i+1}} u\}$.
A weight vector is said to be primitive if it is annihilated by the $e_i$'s.
A {\it highest weight} $U_q(\mathfrak{sl}_{n+1})$-module is a module $\mathcal M$ containing a primitive vector $v$ 
such that $\mathcal M= U_q(\mathfrak{sl}_{n+1}) v$.
The irreducible highest weight module with highest weight $\lambda$ is denoted $V_\lambda$.

Kashiwara~\cite{Kas1,Kas2} introduced a powerful theory whereby combinatorial graphs are used 
to understand finite-dimensional integrable $U_q({\mathfrak sl}_n)$-modules $\mathcal M$.
The {\it crystal} of $\mathcal M$ is a set $B$ equipped with a weight function 
$\wt: B\to \{ x^\alpha : \alpha\in \mathbb Z_{\geq 0}^\infty\}$
and operators $\tilde e_i,\tilde f_i: B\to B\cup\{0\}$ satisfying the properties,
for $a,b\in B$,
\begin{itemize}
\item $\tilde e_i a=b \iff \tilde f_ib=a$
\item $\tilde e_i a=b\implies \wt(b)=\frac{x_i}{x_{i+1}}\wt(a)$.
\end{itemize}
The {\it crystal graph} associated to $B$ is a directed, colored graph 
with vertices from $B$ and edges $E=\{\{a,b\}: b=\tilde e_i(a)\,\text{for some $i$}\}$
labeled by color $i\in I=\{1,\ldots,n-1\}$.  Irreducible submodules are 
in correspondence with {\it highest weight vectors} $b\in B$ where $\tilde e_i(b)=0$ for all $i$.  
These are the vertices of the crystal graph with no incoming edges;
each connected component of $B$ represents an irreducible and contains
a single highest weight vector.  The subset of highest weight vectors $b\in B$ 
where $\wt(b)=x^\gamma$ is denoted by
$\mathbb Y(B,\gamma)$.

The tensor product crystal graph $B_1\otimes\cdots\otimes B_k$
has vertices in the Cartesian product $(b_1,\ldots,b_k)\in B_1\times\cdots\times B_k$ which
are denoted $b=b_1\otimes\cdots\otimes b_k$.  Its weight function  is defined by
$$
\wt(b)=\prod_{j=1}^n \wt_{B_j}(b_j)\,.
$$
A morphism $\Phi: B\rightarrow B'$ is a map on crystal graphs where
$\Phi(0) =0$ and otherwise, for $b\in B$,
\begin{itemize}
\item $\Phi(\tilde f_i(b))=\tilde f_i(\Phi(b))$
\item $\Phi(\tilde e_i(b))=\tilde e_i(\Phi(b))$
\item $\wt(\Phi(b))=\wt(b)$.
\end{itemize}

Lascoux and Sch\"utzenberger anticipated the necessary ingredients for the
Kashiwara type-$A$ crystal in their development of the plactic monoid on words~\cite{LSplactic,LLTplactic}.
It is given by the set 
$$B(1)^n=B(1)\otimes \cdots\otimes B(1)$$
of words in the alphabet 
$B(1)=[n]$; the crystal action $\tilde e_i,\tilde f_i$ is
defined on $b\in B(1)^n$ by changing a single $i$ (or $i+1$) 
to an $i+1$ (or $i$) in the restriction of $b$ to the subword $w_{\{i,i+1\}}$.
Regarding each letter as a parenthesis, $i+1$ as a left and $i$ as a right,
adjacent pairs of parentheses ``()"  are matched and declared to be invisible 
until no more matching can be done.  It is a letter in the remaining subword, $z=i^p(i+1)^q$ 
for some $p,q\in\mathbb Z_{\geq 0}$, which is changed.
Precisely, $\tilde e_i(b)=0$ when $q=0$, 
$\tilde f_i(b)=0$ when $p=0$, and otherwise
$\tilde e_i(b)$ is the word formed from $w$ by replacing the subword $z$ with $i^{p+1}(i+1)^{q-1}$ and
$\tilde f_i(b)$ is formed by replacing $z$ with $i^{p-1}(i+1)^{q+1}$.

\begin{remark}
\label{re:conmaj}
Parentheses pairing of any $b\in B(1)^n$ 
has the property that every adjacent $(i+1)\,i$ is paired, 
and the first $i$ in any adjacent pair $ii$ is never the rightmost unpaired entry.  
Therefore, descents of such pairs are preserved by the action of $\tilde e_i,\tilde f_i$ and
$\des(\tilde e_i(b))=\des(\tilde f_i(b))=\des(b)$ when $b$ is not anhilated.
\end{remark}

%The descent set of elements in $B(1)^n$ thus impose a natural grading on the module.

For $\mu\vdash n$, since $\tilde e_i$ annihilates only the Yamanouchi words,
the set of highest weights of $B=B(1)^n$ with $\wt(b)=x^\mu$ is
\begin{equation}
\label{hwt}
\mathbb Y(B,\mu)=\{b\in B : \text{$b$ is Yamanouchi of weight $\mu$}\}
\,.
\end{equation}
As dictated by Kashiwara's theory,
the crystal graph $B(\mu)$ of the irreducible submodule $V_\mu$ is isomorphic 
to a connected subgraph of $B(1)^n$ which contains a Yamanouchi word of weight $\mu$, and
\begin{equation}
\label{hcrystal}
B \cong \bigoplus_{\mu\vdash n} B(\mu )\times \mathbb Y(B,\mu)\,.
\end{equation}

The crystal graph $B(m)$ is isomorphic to the subgraph of elements $b\in B(1)^m$ with no descents
since $b=(1,\ldots,1)$ is the only element in $\mathbb Y(B(1)^m, (m))$.
Therefore, for any $\gamma\models n$ of length $\ell$, the tensor product crystal
$$
B=B(\gamma_1)\otimes B(\gamma_2)\otimes B(\gamma_\ell)
$$ 
has highest weight elements given by Yamanouchi words which are non-decreasing 
in the first $\gamma_1$ positions, in the next $\gamma_2$ positions, and so forth.

\subsection{Singly graded Garsia-Haiman modules}

A crystal structure on circloids leads us to a characterization for the singly graded decomposition of Garsia-Haiman modules 
which preserves the spirit of the Garsia-Procesi module decomposition given by~\eqref{kftableaux}.

We first refine the decomposition of $B=B(1)\otimes\cdots\otimes B(1)$.  For any
$D\subset\{1,\ldots,n-1\}$, define the induced subposet $B(D)$ of $B$ by restriction to
vertex set $\{ b\in B: \des(b)=D\}$.

\begin{theorem}
\label{the:macq1}
For $\lambda\vdash n$,
\begin{equation}
\label{macq1}
\tilde H_\lambda(x;1,t) = \sum_{D\subset[n-1]} t^{\maj_{\lambda'}(D)}\,
\sum_{\text{highest weight $b\in B(D)$}} s_{\wt(b)}(x)\,,
\end{equation}
where
$$
\maj_\nu(D) = \sum_{i=1}^{\ell(\nu)}\sum_{d\in D\atop |\nu_{<i-1}|<d<|\nu_{<i}|}  \nu_i-d\,.
$$
%\Jennifer{did we define $\nu_{<i}$}
\end{theorem}

\begin{proof}
For any $D\subset [n-1]$, Remark~\ref{re:conmaj} implies that
$B(D)$ is a crystal graph made up of the disjoint union of connected components 
in $B(1)^n$.  Therefore, the Frobenius image of the module associated to $B(D)$ is
$$
\sum_{\text{highest weight $b\in B(D)$}} s_{\wt(b)}(x)\,.
$$
The highest weights are
\begin{equation}
\mathbb Y(B(D),\mu)=\{b\in B(1)^n : \des(b)=D\, \text{and $b$ is Yamanouchi of weight $\mu$}\}
\,.
\end{equation}
More generally, the right hand side of~\eqref{macq1} reflects the graph decomposition of 
the crystal $B(1)^n$ into $B(D)$, graded by $\maj_{\lambda'}(D)$.

On the other hand, Macdonald polynomials at $q=1$ are presented in~\eqref{jimformulabody} 
as weight generating functions of $\lambda$-shaped fillings graded by maj.
Each filling $f\in\F(\lambda,\cdot)$ can be uniquely identified with
a vertex $b\in B(1)^n$ by reading the columns of $f$ from top to bottom
(choosing any fixed column order).  
Consider the filling $f_b$ identified by vertex $b$.   Since
the computation of $\maj(f_b)$ relies only on descents in columns of $f_b$,
precisely the subset of descents involved in the computation of
$\maj_{\lambda'}(\des(b))$, we have that $\maj_{\lambda'}(\des(b)) = \maj(f_b)$.
By definition, $\maj_{\lambda'}(\des(b))$ is constant over all elements $b\in B(D)$
and therefore $\maj(f)$ is constant on all fillings $f$ associated to $b\in B(D)$.
\end{proof}

%Every element of the crystal $B(\mu)$ thus has the same descent set.
%The previous remark implies that elements in a connected component of $B(1)^n$ have
%the same descent set.  Thus, for any $\mu\vdash n$, the induced subposet 
%$B(\mu)=\{b\in B(1)^n: {\rm set}^{-1}(\des(b))=\mu\}$
%is a crystal graph for the irreducible $V_\mu$.
%

\begin{remark}
\label{re:majnotconstant}
Although each vertex $b\in B(1)^n$ could be uniquely identified with the filling $f$ 
of shape $\lambda\vdash n$ whose {\it reading word} is $b$,
maj is not constant on all fillings in the same connected component
under this correspondence.
\end{remark}

\label{circloidcrystal}

The interaction of crystals with the cocharge statistic comes out of a directed, colored graph $\mathcal B(\gamma)$ 
whose vertices are circloids of weight $\gamma$.  An $i$-colored edge between circloids
is imposed by operators, $\tilde e_i$ and $\tilde f_i,$ which move an entry from sector $i+1$ 
to sector $i$ or vice versa using a method of pairing colored letters.

Pairing is a process which iterates over each entry in a given sector.
Entries are considered from smallest to largest with respect to the 
{\it co-prismatic order}, defined on colored letters by
$$
x_y>' u_v \iff
y>v\quad\text{or}\quad
y=v\;\text{and}\; x>u\,.
$$
Pairing is done by writing the entries from sectors $i$ and $i+1$ in co-prismatic decreasing order, assigning every
entry from sector $i+1$ a left parenthesis and every entry from sector $i$ a right parenthesis.  Entries are then
paired as per the Lascoux and Sch\"utzenberger rule for parenthesis.

\begin{definition}
For a composition $\alpha$ and $i\in\{1,\ldots,\ell(\alpha)-1\}$, the operator $\tilde e_i$ acts on $C\in\C(\cdot,\alpha)$ 
by moving the largest unpaired entry in sector $i+1$ to the unique position of sector $i$ which 
preserves the prismatic decreasing condition on circloids.  In contrast, $\tilde f_i$ acts on $C$ by moving the 
smallest unpaired entry in sector $i$ to sector $i+1$.
% and 
%$$
%\tilde s_i=\begin{cases} \tilde e_i^a & \text{if $\alpha_{i+1}>\alpha_i$}\\
%\tilde f_i^a &\text{otherwise}\,,
%\end{cases} 
%$$
%where $a=|\alpha_{i+1}-\alpha_i|$.
\end{definition}

\begin{remark}
For $\alpha,\gamma\models n$ and a circloid $C\in \C(\gamma,\alpha)$, 
if $\tilde e_i$ does not anhilate $C$ then
its action preserves weight since it involves only moving an entry;
the entry is moved from sector $i+1$ to sector $i$, implying that
$\tilde e_i(C)\in\C(\beta,\alpha)$
where $\beta=(\gamma_1,\ldots,\gamma_{i-1},\gamma_{i}+1,\gamma_{i+1}-1, \ldots)$.
\end{remark}

\begin{theorem}
For any composition $\gamma\models n$, $\mathcal B(\gamma)$ is a crystal graph and its highest weights
(of weight $\mu\vdash n$) are in bijection with
$$
\SSCT(\cdot,\mu)= \{T\in\CT(\cdot,\mu) : \text{colored letters increase up columns of $T$, with respect to $<'$}\}\,.
$$
When $\gamma$ is a partition, the connected components of $\mathcal B(\gamma)$ are constant on $\cocharge$.
\end{theorem}

\begin{proof}
Given $\gamma\models n$,
let $\phi_i$ act on $\F(\gamma,\cdot)$ by the induced action 
$\phi_i=\mathfrak f\circ\tilde e_i\circ\mathfrak f^{-1}$.
When an entry $x_y$ in circloid $C$ is paired with $u_v<'x_y$ by the action of $\tilde e_i$,
an $i+1$ in cell $(x,y)$ of $\phi_i(f^{-1}(C))$ is paired with an $i$ in cell $(u,v)$ where 
either $y>v$ or ($y=v$ and $x>u$).  Consider the graph on $\F(\gamma',\cdot)$ where an $i$-colored 
edge connects $f$ and $\hat f$ when $\hat f'=\phi_i(f')$; 
when each filling is replaced by its reading word $b$,
this is the crystal graph $B(1)^n$.  In particular, 
a crystal morphism $\Phi:\mathcal B(\gamma)\to B(1)^n$ 
is given by
$$
\Phi(C)= b\,,
$$
where $b$ is the reading word obtained by reading {\it down the columns} of the filling $\mathfrak f(C)$ 
from {\it right to left}.  That is, cell $(x,y)$ of $\mathfrak f(C)$ is read before cell $(u,v)$ if $y>v$ or ($y=v$ and
$x>u$).  This is equivalent to $x_y>'u_v$.
The weight function on $\mathcal B(\gamma)$ maps circloid $C$ to its shape
by~\eqref{codomainf}.

A highest weight $C\in \mathcal B(\gamma)$ satisfies $\tilde e_i(C)=0$ for all $i$ if and only if
each entry $x_y$ in row $i+1$ of $T=\iota(C)\in\CT(\mu,\cdot)$ pairs with 
an entry $u_v<' x_y$ in row $i$.  By rearranging the rows of $T$ so that they are co-prismatically non-decreasing
this will result in co-prismatic increasing columns.
Note that $\mu$ must be a partition, for if a sector $i+1$ of $C$ has more entries than 
sector $i$, then $C$ has an unpaired entry and is not anhilated.
\end{proof}

The circloid crystal captures a formula for the $q=1$ Macdonald polynomials which is 
perfectly aligned with the long-standing formula for $q=0$ given by Lascoux-Sch\"utzenberger.

\begin{corollary}
For any partition $\mu$,
$$
\tilde H_\mu({x};1,t)= 
\sum_{T\in\SSCT(\cdot,\mu)}
t^{\cocharge(T)}\,s_{\shape(T)} ({x})
\,.
$$
\end{corollary}
\begin{proof}
We have seen that each $b\in B(D)$ corresponds to a filling $f_b$ with $\maj_{\lambda'}(D)=\maj(f_b)$.
Theorem~\ref{the:macq1} gives that
$$
\tilde H_\lambda(x;1,t)=\sum_{D\subset[n-1]}\sum_{\mu\vdash n}\sum_{b\in\mathbb Y(B(D),\mu)} t^{\maj(f_b)}\,s_{\mu}({x})
=\sum_{\mu\vdash n}\sum_{b\in \mathbb Y(B(1)^n,\mu)}  t^{\maj(f_b)}\,s_{\mu}({x})\,.
$$
The claim follows by recalling  that $\maj(f_b)=\cocharge(\mathfrak f^{-1}(f_b))$ and that
$\Phi$ is a morphism of crystals.
\end{proof}

%\begin{proof}
%Consider a highest weight element $b\in B(1)^n$ and its associated statistic $\maj_\mu(\des(b))$.  
%Since $b$ is Yamanouchi if and only if the transpose $f'$ of $f_b$ is Yamanouchi, we have that
%$\hat T=\mathfrak c^{-1}(f')$ has prismatic increasing columns by Proposition~\ref{prop:yamacs}.
%However, $\maj(f')\neq \maj(f_b)$.
%
%
%Define a map $\mathfrak g$ on colored letters by
%\[ \mathfrak g(x_y)=(\mu_x-y+1)_x. \]
%Given a colored tabloid $T$, let
%$\mathcal E_i = \{g(x_y): x_y\text { lies in sector $i$ of $C$}\}$.
%Define the colored tabloid $\hat T=\mathfrak g(T)$ by placing
%entries of $\mathcal E_i$ increasing in  row $i$ 
%with respect to transposed prismatic order.
%In particular, the transpose of $\hat T$ is prismatic increasing in columns.
%
%
%Since $x_y \leq u_v$ if and only if $\mathfrak g(x_y) \leq' \mathfrak g(u_v)$,
%$T$ is prismatic increasing in columns if and only if $\mathfrak g(T)$ is transpose 
%prismatic increasing in columns.  
%
%Recall by Proposition~\ref{prop:yamacs} that $f=\mathfrak c(T)$ is Yamanouchi 
%if and only if $T$ is prismatic increasing in columns.  
%Further, $C$ is prismatic increasing in columns if and only if $g(C)$ is transpose prismatic increasing in columns.  
%However, the action of $g$ on $f$ is to conjugate and then reverse the columns.  
%Therefore $f$ is Yamanouchi if and only if $\mathfrak f(g(C))$ is column Yamanouchi.
%\end{proof}

From this, it is not difficult to rederive Macdonald's formula
taken over standard tableaux.

\begin{corollary}\cite{Macbook}
For any partition $\mu\vdash n$,
$$
\tilde H_\mu({x};1,t)
= \sum_{T\in\SSYT(\cdot,1^n)}
\prod_{i=1}^{\mu_1-1}t^{\cocharge(T_i)}\,
s_{\shape(T)}({x}) \,
$$
where $T_i$ is the subtableau of $T$ restricted to letters in $[\mu'_i+1,\mu'_{i+1}]$.
\end{corollary}
\begin{proof}
Give a prismatic column increasing  circloid $C\in\C(\cdot,\mu)$,
replace each entry $i_c$ of $C$ with letter $i+\sum_{j<c}\mu_j'$.
The condition on $C$ that $x<u$ or $x=u$ and $y<v$ 
for any $x_y$ above $u_v$ implies that letters 
are strictly increasing in columns of the tabloid $T$.
Since the computation of cocharge on a circloid independently
calculates cocharge on standard subwords of a given color,
$\mu-\cocharge(T)=\cocharge(C)$.
\end{proof}

%\jennifer{we have an issue using wt versus weight.  I think wt should only be used for the map
%on elements in a crystal and weight should be used for weight of words, tableaux, etc.  Do you agree?
%If so, will you check that we have always used weight when appropriate?
%Also, we consider strict compositions as the shape of a circloid.  Are we specifying strict when 
%necessary?  This is another possible ambiguity that needs to be made more clear.}

\subsection{Double crystal structure}

Characterization of the doubly graded irreducible decomposition of Garsia-Haiman modules presents major obstacles.  
Although the identification of fillings with elements of $B(1)^n$ given by column reading yields 
connected components constant on the maj-statistic, it is incompatible with the inversion triples.
Even the subset of vertices with zero inversion triples is not a connected component.
The crystal cannot be applied to~\eqref{jimformula}, even when $q=0$, to gain insight on bi-graded
decomposition of $\mathcal R_\mu(x,y)$ into its irreducible components.
However, a double crystal structure using dual Knuth relations (jeu-de-taquin) and $\tilde e_i,\tilde f_i$ operators
on colored tabloids can be applied to the Garsia-Procesi modules.
Double crystal structures on $B(\mu_1)\otimes\cdots \otimes B(\mu_\ell)$ 
have been studied in various contexts~\cite{VanL,Sdummy,Ldouble}, but without
regard to graded modules.

For any composition $\gamma$ of length $\ell$, we consider a crystal
$\mathcal B^\dagger(\gamma)$ on vertices $\CT(\cdot,\gamma)$ which is dual to $\mathcal B(\gamma)$.
An $i$-colored edge is prescribed by a {sliding operation} defined on an inflation of rows $i$ and $i+1$ in
a colored tabloid.
The {\it $i$-inflation} of a vertex $b\in \mathcal B^\dagger(\gamma)$ is defined by 
spacing out the colored letters in row $i$ of $b$ while preserving their relative order as follows:
entries $e$ are taken from west to east from row $i$ of $b$ and placed in the leftmost empty cell 
of row $i$ without an entry $e'<e$ directly above it.  
The operator $e_i^\dagger$ on $b\in \mathcal B^\dagger$ is then defined by 
a jeu-de-taquin {\it sliding} action whereby the largest entry of row $i+1$ 
in the $i$-inflation of $b$ which lies immediately above an empty cell is swapped 
with this empty cell, after which all empty cells are removed.
When no empty cell lies in row $i$, $e_i^\dagger(b)=0$.

It is convenient to define the {\it inflation} a vertex $b\in \mathcal B^\dagger(\gamma)$
as the punctured colored tabloid obtained by inflating rows of $b$ in succession from top
to bottom.  Note that the inflation of $b$ has entries (prismatic order) increasing up columns.  
%Edges are prescribed by a {sliding operation} defined on inflated colored tabloids.
%The {\it inflation} of a vertex $b\in \mathcal B^\dagger(\gamma)$ is a punctured colored tabloid with entries 
%(prismatic order) increasing up columns.  It is defined by a new placement 
%of colored letters from row $r$ of $b$ into row $r$ of $\mathbb N\times \mathbb N$.  
%The top row $\ell$ is simply the top row of $b$.  Squares of row $r=\ell-1$ are then filled with entries 
%from row $r$ of $b$, in the same relative order, as follows; entries $e$ are taken from west to east from row $r$ of $b$ 
%and placed in the leftmost empty cell of row $r$ without an entry $e'<e$ directly above it.
%Iterate, placing letters from row $r-1$ of $b$ into row $r-1$.
%Define the $i$-inflation of a colored tabloid by inflating only the $i\th$ row with respect to row $i+1$.
\begin{example}
The $2$-inflation of 
$T=\tableau[scY]{1_1,2_3,3_2|1_2,2_1,3_3|1_3,2_2,3_1}\,$ is
$\,\tableau[scY]{1_1,2_3,3_2|1_2,,2_1,3_3|1_3,2_2,3_1}.\, $
The inflation of $T$ is $\, \tableau[scY]{1_1,2_3,3_2|1_2,,2_1,3_3|1_3,,2_2,,3_1}$.
\end{example}

%\begin{remark}  The $e_i^\dagger$ operators are defined similarly on the inflation of $b$;
%if an empty cell is introduced to the top row, simply slide it to the right until it is out of the diagram.  This
%can always be done without changing the shape of the rest of the rows because the entries to the right of
%the empty cell must be strictly larger than the moved entry, otherwise the moved entry would not have had
%an empty cell immediately beneath.
%%\jennifer{I'm not following this?}
%\end{remark}

\begin{definition}
For any $\gamma\models n$, let $\mathcal B^\dagger(\gamma)$ be the graph on vertices
$\CT(\cdot,\gamma)$ with a directed, $i$-colored edge from $b\to b'$ when $e_i^\dagger(b)=b'$.
\end{definition}

We will establish that $\mathcal B^\dagger(\gamma)$ is a crystal graph doubly related to $B(1)^n=B(1)\otimes\cdots\otimes B(1)$
by the companion map.  For each $\gamma\models n$, define the map $\mathfrak c_\gamma$ on 
$B(1)^n$ by
$$
\mathfrak c_\gamma (b) = \mathfrak c( f_b)\,,
$$
where $f_b$ is the unique filling of shape $\gamma$ whose reading word
is $b$.

\begin{theorem}
\label{the:daggerops}
For each $\gamma\models n$, $\mathfrak c_\gamma$ is a crystal isomorphism 
$$
B(1)\otimes\cdots\otimes B(1)\cong \mathcal B^\dagger(\gamma)
\,.  
$$
\label{prop:hwt}
The highest weights in $\mathcal B^\dagger(\gamma)$ of weight $\mu\vdash n$ are
$$
\mathbb Y(\mathcal B^\dagger(\gamma),\mu) = \{T\in \CT(\mu,\gamma): \text{entries of $T$ prismatically increase up columns}\}\,.
$$
\end{theorem}

\begin{proof}
Since the image of the companion bijection $\mathfrak c$ on $\F(\gamma,\cdot)$ is the set of vertices in $\mathcal B^\dagger(\gamma)$,
the map $\mathfrak c_\gamma$ is a bijection between $B(1)\otimes\cdots\otimes B(1)$
and the vertices of $\mathcal B^\dagger(\gamma)$.  To check that edges match,
we need to prove that $\tilde e_i(b)=b'$ if and only if
$\tilde e_i^\dagger(\mathfrak c_\gamma(b))= \mathfrak c_\gamma(b')$,
for each $b\in B(1)\otimes\cdots\otimes B(1)$.

We will show that an entry $x_y$ in row $i+1$ of the $i$-inflation of $\mathfrak c_\gamma(b)$ has an empty cell under it if and only if the corresponding $i+1$ in cell $(x,y)$ of $b$ is unpaired.  The proof then follows because sliding the rightmost  $x_y$ down to row $i$ is the equivalent of changing the leftmost unpaired $i+1$ to an $i$ in $b$.

Suppose that $x_y$ lies above an empty cell in the $i$-inflation of $\mathfrak c_\gamma(b)$.  Then for each $u_v<x_y$ in row $i$ the entry $u'_{v'}$ immediately above in row $i+1$ satisfies $u_v<u'_{v'}<x_y$.  The entry $x_y$ corresponds
to an $i+1$ in cell $(x,y)$ of $b$, and for every $i$ appearing afterward, there is a distinct $i+1$ appearing between it and cell $(x,y)$.  Therefore the $i+1$ in cell $(x,y)$ will be unpaired.

Suppose that an $i+1$ in cell $(x,y)$ of $b$ is unpaired.  Then every $i$ appearing afterward is paired with an $i+1$ that appears between it and cell $(x,y)$  Therefore, for each $u_v$ in row $i$ of $\mathfrak c_\gamma(b)$ with $u_v<x_y$, there is a unique $u'_{v'}$ in row $i+1$ such that $u_v<u'_{v'}<x_y$.  Thus we are guaranteed that there is an empty cell under $x_y$ when we $i$-inflate $\mathfrak c_\gamma(b)$.

The highest weights of $B(1)^n$ are defined by the Yamanouchi property 
and thus their companions are prismatic column increasing by Proposition~\ref{prop:yamacs}.
Alternatively, $e_i^\dagger$ anhilates a colored tabloid $T$ when there is no entry in row $i+1$ of
the inflation of $T$ above an empty square in row $i$.  In particular, $T$ is its own inflation and
thus has prismatic increasing columns.
\end{proof}

\subsection{Garsia-Procesi modules}

As a first application, we show how the graded irreducible decomposition of a Garsia-Procesi module is readily apparent 
in the crystal $\mathcal B^\dagger$.  For this, we need only the induced subposet $\mathcal B_0^\dagger(\gamma)$ 
on the restricted set of reverse-colored vertices in the crystal graph $\mathcal B^\dagger(\gamma)$.

\begin{prop}
\label{prop:hwtss}
For a composition $\gamma$ of length $\ell$,
the companion map is a crystal isomorphism
$$
\mathcal B_0^\dagger(\gamma)\cong B(\gamma_1)\otimes\cdots\otimes B(\gamma_\ell)\,.
$$
The highest weights of $\mathcal B_0^\dagger(\gamma)$ are (reverse-colored) semi-standard tableaux
of weight $\gamma$.
\end{prop}
\begin{proof}
Given $b\in \mathcal B_0^\dagger$, consider $b'=\tilde e_i^\dagger(b)$
If $b'\neq 0$, there is a (largest) unpaired $x_y$ in row $i+1$ of the 
inflation ${\bf b}$ above an empty cell.  The definition of inflation
thus implies $x_{y-1}$ cannot lie in row $i$.  Therefore, the image of a 
reverse colored tabloid under the crystal action remains as such
since the action merely slides $x_y$ into row $i$.
That is, each connected component in $\mathcal B_0^\dagger$ is a connected component of
$\mathcal B^\dagger$.
We then note that the set of vertices $\mathcal B_0^\dagger$  is in bijection with 
$B(\gamma_1)\times\cdots\times B(\gamma_\ell)$
since reverse-colored tabloids of weight $\gamma$ are the companion images of 
tabloids with shape $\gamma$ by Proposition~\ref{prop:revcs}. 
In turn, $B(\gamma_1)\times\cdots\times B(\gamma_\ell)$ are defined as induced
subposets of $B(1)^n$ allowing us to apply Theorem~\ref{prop:hwt} to establish the
isomorphism.

The highest weights of $B^\dagger(\gamma)$ are the colored tabloids with 
prismatically increasing columns  by Theorem~\ref{prop:hwt}. 
Thus, a highest weight element $b$ has entries
$x_y$ above $x_{y'}$ in the same column only when $y<y'$.  However, if $b$ is also
reverse colored, then $y>y'$ and therefore its {\it letters} increase up columns.
\end{proof}

Define the {\it faithful recoloring} of $b\in \mathcal B^\dagger$ to be the colored tabloid $b_0$ obtained by
stripping $b$ of its colors and then faithfully coloring its letters.
%Note that $\cocharge(b_0)$ can be computed by ignoring colors of $b$ and taking the Lascoux-Sch\"utzenberger cocharge.

\begin{lemma}
\label{lem:eioncocharge}
For $\mu\vdash n$ and $b\in \mathcal B^\dagger(\mu)$ with the property that 
letters increase up columns of the inflation of $b$,
$$
\cocharge(b_0)=\cocharge((b')_{0})\quad
\text{for $\;b'=\tilde e_i^\dagger(b)\neq 0$}.
$$
\end{lemma}

\begin{proof}
Consider $b\in \mathcal B^\dagger$ such that {letters} increase up columns in 
its inflation ${\bf b}$.  If $b'=\tilde e_i^\dagger(b)\neq 0$, then there is a largest entry $x_y$ 
in row $i+1$ of ${\bf b}$ which lies above an empty cell and it is moved into row $i$ to
obtain $b'$ from $b$.  Since {\it letters} must increase up columns of ${\bf b}$ and an $x$ in
row $i+1$ lies above an empty cell, there must be fewer $x-1$'s in row $i$ than there are $x$'s 
in row $i$ of $b$ by construction of ${\bf b}$.  For the number $n_x^i$ of letters smaller than $x$
in row $i$, $n_{x}^i>n_{x-1}^{i-1}$.

Let $d$ be the smallest color adorning an $x$ in row $i+1$ of $b_0$ and note that
there can be no $(x-1)_{d'}$ with $d'<d$ in row $i$ of $b_0$ by definition of 
faithful coloring.  If $(x-1)_{d}$ does not lie in row $i$ of $b_0$, then 
$b_0$ and $(b')_0$ are the same with the exception of one entry;
$x_{d}$ lies in row $i+1$ of $b_0$ and in row $i$ of $(b_0)'$.
Therefore, their cocharges equal since $(x-1)_{d}$ is not in row $i$ of either element.

Otherwise, $d\in D\cap E$ for the set $D$ of colors adorning letter
$x$ in row $i+1$ and the set $E$ of colors attached to $x-1$ in row $i$ of $b_0$.
In this case, $b_0$ and $(b_0)'$ again differ only by one entry; $x_{e}$ lies in row $i$
of $b_0$ and in row $i-1$ of $(b')_0$ for the smallest element $e\in D\backslash E$.
Therefore, their cocharges match since $(x-1)_e$ is not in row $i$.
\end{proof}

\begin{prop}
For $\mu\vdash n$, the graded decomposition of the Garsia-Procesi module $\mathcal R_\mu(y)$ into 
its irreducible components is encoded in the crystal graph $\mathcal B_0^\dagger(\mu)$ with
$\cocharge(b_0)$ attached to each vertex $b$.
\end{prop}
\begin{proof}
The Frobenius image of the module $\mathcal R_\mu(y)$ is obtained by setting $q=0$ in 
our formula~\eqref{rjformula}.  We then apply $\iota$
to obtain 
\begin{equation}
\label{hallinB}
\tilde H_\mu(x;0,t)=\sum_{T}
 t^{\cocharge(T)}\,x^{\shape(T)}\,,
\end{equation}
over faithfully colored tabloids $T$ of weight $\mu$.
The map $b\mapsto b_0=T$ is a bijection between $\mathcal B_0^\dagger(\mu)$ and $\CT(\cdot,\mu)$
since reverse colored and faithfully colored tabloids are both manifestations of uncolored tabloids. 
This allows us to convert the previous identity to
\begin{equation}
\tilde H_\mu(x;0,t)
=\sum_{b\in \mathcal B_0^\dagger(\mu)} t^{\cocharge(b_0)}\,x^{\shape(b)}
=\sum_{b\in \mathcal B_0^\dagger(\mu)} t^{\cocharge(b_0)}\,x^{\wt(b)}
\,,
\end{equation}
recalling that the weight function on the crystal $\mathcal B_0^\dagger(\mu)$ sends $b\mapsto \shape(b)$.
Now, letters increase up columns in the inflation of a reverse colored tabloid $b$; 
if $x_y$ lies in a higher row than $x_z$, then $x_y<x_z$ and thus these entries cannot 
lie in the same column of the inflation of $b$.  This given, Lemma~\ref{lem:eioncocharge} implies that
attaching each vertex $b$ to ${\cocharge(b_0)}$ 
is a statistic which is constant on the connected components of $\mathcal B_0^\dagger(\mu)$.  Therefore,
\begin{equation}
\tilde H_\mu(x;0,t)
= \sum_{b\in \mathbb Y(\mathcal B_0^\dagger(\mu),\lambda)} t^{\cocharge(b_0)}\,s_{\lambda}(x)
= \sum_{T\in\SSYT(\lambda,\mu)} t^{\cocharge(T)}\,s_{\lambda}(x)\,,
\end{equation}
since Proposition~\ref{prop:hwtss} characterizes the highest weights by
(reverse colored) semi-standard tableaux.
%and noting that cocharge of a
%faithfully colored tabloid $b_0$ matches the Lascoux-Sch\"utzenberger cocharge 
%of its manifest uncolored tabloid.
\end{proof}

% this is H_{1^n} in terms of \mu-cocharge where \mu=(1^n)

\section{Further applications of the double crystal}

\subsection{Energy function on affine crystals}

Lenart and Schilling~\cite{LenSch} connect the $q=0$ Macdonald polynomials to a tensor product 
of Kashiwara-Nakashima single column crystals, producing a new statistic for computing 
the (negative of the) energy function on affine crystals.  The double crystal reveals that
their statistic is precisely the companion of cocharge.

For a partition $\lambda$ of length $\ell$
and $b\in B(\lambda_1)\otimes \cdots\otimes B(\lambda_\ell)$,
we define
$$
\zigmaj(b)=\sum_{w\in Z(b)} \maj(w)\,,
$$
where $Z(b)$ is a set of $\lambda_1$ words extracted from the letters of 
$b=(b_1,\ldots,b_\ell)$. 
The first word $w=w_{\lambda_1'}\cdots w_1$ is constructed by selecting $w_1$ 
to be the smallest entry in $b_1$.
Iteratively, $w_i$ is selected to be the smallest entry larger than $w_{i-1}$
in $b_i$ (breaking ties by taking the easternmost). If there is no larger entry 
available, the smallest entry in $b_i$ is selected instead.
The first word $w$ is fully constructed after a letter has been selected 
from $b_\ell$.  The remaining words of $Z(b)$ are constructed by
the same process, ignoring previously selected letters of $b$.

\begin{example}
For $b=(223,123,12)\in B(3)\otimes B(3)\otimes B(2)$,
the words in $Z(b)$ are
$Z(b)=\{132, 212, 23\}$ implying that $\zigmaj(b)=2+1+0$.
\end{example}

\begin{prop}
\label{prop:zmaj}
For partition $\lambda$ of length $\ell$ and $b\in B(\lambda_1)\otimes\cdots\otimes B(\lambda_\ell)$, 
$$
\zigmaj(b)= \cocharge(\mathfrak c(f_b)_0)\,,
$$
where $f_b$ is the unique $\lambda$-shaped tabloid with reading word $b$.
\end{prop}
\begin{proof}
For $b\in B(\lambda_1)\otimes\cdots\otimes B(\lambda_\ell)$, 
consider first the case that the tabloid $f_b$ contains an $x$ in row $i$ 
and letter $z>x$ in row $i+1$.  For the colored tabloid $T=\mathfrak c(f_b)$,
there is an $i$ in row $x$ of $T$ and and $z$ is the lowest row above $x$ with 
an $i+1$.  If $f_b$ does not have any $z>x$ in row $i+1$, instead take $z\leq x$
to be the minimal entry in row $i+1$ and note that $z$ is the lowest row in $T$ 
containing an $i+1$.  
More generally, for the $i\th$ word $w=w_{\lambda_i'}\cdots w_1$ in $Z(b)$,
$w_j$ records the row containing $j_i$ in $T$.  Since
each $w_{j+1}>w_j$ contributes $\lambda_i'-j$ to the $\maj$
and each $j+1$ higher than $j$ contributes the same to $\cocharge(T)$
the claim follows.
\end{proof}

Lemmas~\ref{lem:eioncocharge} and Proposition~\ref{prop:zmaj} 
imply that connected components of 
the crystal graph $ B(\lambda_1)\otimes \cdots\otimes B(\lambda_\ell)$
are constant on $\zigmaj$.

\begin{cor}
\label{cor:zigcrystal}
For a partition $\lambda$ of length $\ell$ and $b\in  B(\lambda_1)\otimes \cdots\otimes B(\lambda_\ell)$,
$$
\zigmaj(b)=\zigmaj(\tilde e_i(b))\qquad \text{and} \qquad
\zigmaj(b)=\zigmaj(\tilde f_i(b))\,,$$ 
for any $i\in\{1,\ldots,n-1\}$.
\end{cor}

\begin{theorem} 
\label{the:H2zigs}
For any partition $\lambda$,
$$
\tilde H_\lambda(x;0,t)=
\sum_{b\in B(\lambda_1)\otimes\cdots\otimes B(\lambda_\ell)
\atop
b~Yamanouchi } t^{\zigmaj(b)}\,s_{\weight(b)}({x})
\,.
$$
\end{theorem}
\begin{proof}
The expansion~\eqref{hallinB} of $\tilde H_\mu(x;0,t)$ over elements of
$B^\dagger$ can be converted to one involving $B(\lambda_1)\otimes\cdots\otimes B(\lambda_\ell)$ 
by Proposition~\ref{prop:hwtss}.  
As reviewed in~\eqref{hwt}, the highest weights of $B=B(\lambda_1)\otimes\cdots\otimes B(\lambda_\ell)$ 
are characterized by the Yamanouchi
condition.   Corollary~\ref{cor:zigcrystal} implies that $\zigmaj$ is constant on
connected components of $B$, which are in correspondence with Schur functions indexed 
by the the highest weights.
\end{proof}

\subsection{Zero inversion map on Macdonald fillings}
The faithful recoloring of vertices in $B^\dagger_0$ not only gives
a formula for the energy function, it exposes
an identification between inversionless fillings and tabloids used
in~\cite{HHL}.  Define 
$$
\mathfrak s: \{F\in\F(\lambda,\alpha) :inv(F)=0\} \to \T(\lambda,\alpha)
$$
on a filling $F$ simply by rearranging entries in each row
into non-decreasing order from west to east.  
The inverse of $\mathfrak s$
is defined in~\cite{HHL}(Proof of Proposition 7.1) by 
uniquely constructing an inversionless filling from a collection 
of multisets, $m=\{m_1,m_2,\ldots, m_k\}$.
The unique placement of entries from $m_i$ into row $i$ of $f$ 
so that $\inv(f)=0$ requires first that entries of $m_1$ are put into the bottom
row in non-decreasing order, from west to east.  Proceeding to the next row $r=2$, 
letters of $m_r$ are placed in columns $c$ from 
west to east as follows: an empty cell $(r,c)$ is filled with the smallest value
that is larger than the entry in $(r-1,c)$.  If there are
no values remaining that are larger than that in $(r-1,c)$, 
the smallest available value is chosen.  The filling $f$ arises
from iteration on rows and by construction, $f$ has no inversion triples.

In fact, $\mathfrak s^{-1}(f)$ is none other than the companion preimage of the faithful 
recoloring of $f$'s companion.  That is,
the companions of $f$ and $\mathfrak s(f)$ are both manifestations
of the same tabloid, one is reverse colored and the other is faithful.

\begin{prop}
For any tabloid $f\in\T(\gamma,\cdot)$,
$$
\mathfrak c(\mathfrak s^{-1}(f))=\mathfrak c(f)_0\,.
$$
\end{prop}
\begin{proof}
For any tabloid $f$, $\mathfrak s^{-1}(f)$ and $f$ differ only by the
rearrangement of entries within rows.  Thus, by definition of companions,
$T'=\mathfrak c(\mathfrak s^{-1}(f))$ and $T=\mathfrak c(f)$ differ 
only by their colorings.  
Since $f$ is a tabloid, $T$ is reverse colored by
Proposition~\ref{prop:revcs} and since $\mathfrak s^{-1}(f)$ is inversionless,
$T'$ is faithfully colored by Proposition~\ref{prop:charge2maj}.
Each of these colorings is uniquely defined on the manifest tabloid
and the claim follows.
\end{proof}

%\begin{example}  Consider the multisets $m_1 = \{2,2,3\}, m_2=\{1,2,3\},$ and $m_3=\{1,1,3\}$.  The
%only way to make those the rows of a filling without inversion triples is
%\[ \tableau[scY]{ 1,3,1 | 3,1,2 | 2,2,3 }\;. \]
%Then if we make the $i\th$ $\zigmaj$ subsequence the $i\th$ column, we see that it is equivalent
%to computing $\mathfrak s(T)$:
%\[ \tableau[scY]{ 1_1,3_2,1_3 | 3_1,1_2,2_3 | 2_1,2_2,3_3 }\;. \]
%\end{example}

\subsection{$K$-theoretic implications}

\label{sec:ktheory}

To give a flavor of how circloid crystals fit into $K$-theoretic Schubert calculus,
consider tabloids with the property that their conjugate is also a tabloid.
Such a filling is called a {\it reverse plane partition}.
If the {\it weight} of a reverse plane partition is defined to be the
vector $\alpha$ where $\alpha_i$ records the number of columns containing
an $i$, the weight generating functions are representatives for
$K$-homology classes of the Grassmannian:
for skew partition $\nu/\lambda$,
$$
g_{\nu/\lambda}(x) = \sum_{r\in \RPP(\nu/\lambda,\cdot)} x^{\weight(r)}\,.
$$

From this respect, repeated entries in a column of the reverse plane partition $r$ 
are superfluous motivating us to instead identify $r$ with the tabloid $f$
obtained by deleting any letter that is not the topmost in its column and then 
left-justifying letters in each row.  The {\it inflated shape} of tabloid $f$ 
is defined to be the shape of its inflation.  This recovers the shape of the
reverse plane partition $r$ from whence $f$ came.

\begin{prop}
\label{prop:ktheory}
For skew partition $\nu/\lambda$,
\begin{equation}
\label{ktheory}
g_{\nu/\lambda}(x) = \sum_{Yamanouchi~tabloid~T\atop
{inflated shape}(T)=\nu/\lambda}
s_{\weight(f)}(x)\,.
\end{equation}
\end{prop}
\begin{proof}
For any composition $\gamma$, since $B=B(\gamma_1)\otimes\cdots\otimes B(\gamma_\ell)$
is a crystal graph under $\tilde e_i$, it suffices to show that
the inflated shape of $f\in\T(\gamma,\cdot)$ is the same as the inflated
shape of $\tilde e_i(f)$.  From this, the induced subposet of $B$
on vertices with fixed inflated shape is also a crystal and the
Schur expansion comes from the highest weights.

Consider a filling $f$ and $f'=\tilde e_i(f)$ differing by only one 
letter $i+1$ changed to an $i$ in some row $r$.  The shape of the 
inflation of $f$ can differ from that of $f'$ only if there is an 
$i+1$ in row $r+1$ lying above an empty cell.
However, since the leftmost unpaired $i+1$ in $f$ lies in row $r$,
the process of pairing ensures that every $i+1$ in row $r+1$ is 
paired with an $i$ in row $r$.  Therefore, there are more $i$'s in 
row $r$ of $f$ than there are $i+1$'s in row $r+1$ and the
inflation of $f$ must have an entry smaller than $i+1$ below
the rightmost $i+1$ in row $r+1$.
\end{proof}

An expression for the Schur expansion of $g_{\nu/\lambda}$ over a sum of semi-standard 
tableaux arises as a corollary of Proposition~\ref{prop:ktheory} by applying the companion map
to~\eqref{ktheory} and using Corollary~\ref{cor:revcompanion}.  Such an expression opens
up the study of problems in $K$-theoretic Schubert calculus to the classical theory of
tableaux.  For example, a simple bijective proof of the $K$-theoretic Littlewood-Richardson
was given in~\cite{LMS} using this approach.

\section{Quasi-symmetric expansion}

It is not difficult to use dual equivalence graphs instead of crystals to deduce
our previous results.  For this, we formulate Macdonald polynomials using
colored words, without mention of shape, in terms of
Gessel's {\it fundamental quasisymmetric function}.  
Defined for any $S\subseteq[n]$, let
$$
Q_S(x)=
\sum_{\substack{i_1\leq\cdots \leq i_n\\ i_j<i_{j+1} \text{ iff } j\in S}}
x_{i_1}^{\beta_1}\cdots x_{i_n}^{\beta_n}\,.
$$

%The statistics maj and inv are defined on permutations $w\in S_n$ by 
%considering $w$ as the unique filling of column shape $(1^n)$ whose reading word
%is $w$.  

Betrayal and cocharge are defined on a colored word $w$ by computing these statistics
on the circloid $C$ obtained by writing entries of $w$ counter-clockwise on a circle.
Since the shape of a circloid has no bearing on the statistics, it suffices to write
each entry in its own sector so that $C$ has shape $(1^{\ell(w)})$.

\begin{theorem}
For partition $\mu\vdash n$,
$$
\tilde H_\mu(X;q,t) = 
\sum_{colored~word~w\atop \weight(w)=\mu}
q^{\ski(w)}\, t^{\cocharge(w)}\, 
Q_{\des_n(w)}(x)\,,
$$
where $\des_n(w)=\{n-d:d\in \des(w)\}.$
\end{theorem}
\begin{proof}
Given a fixed colored word $w$ of weight $\mu$, consider the
set of weak compositions $\beta$ for which  $\shape(w)=\beta$.
For each such $\beta$, $\des(w)\subseteq{\rm set}(\beta)$.
Therefore, a unique circloid $C\in \C(\beta,\mu)$
is obtained by counter-clockwise inscribing the entries of $w$ 
on a circle, separated into sectors of sizes $\beta_\ell,\ldots,\beta_1$.
Since the computation of cocharge and betrayal of
circloid $C$ does not involve its shape, every $C$ arising in this
way satisfies
$\cocharge(C)=\cocharge(w)\;\text{and}\;\ski(C)=\ski(w)$.
Theorem~\ref{the:rjformula} can thus be rewritten as
$$
\tilde H_\mu(X;q,t) = 
\sum_{C\in\C(\cdot,\mu)}
q^{\ski(C)}\, t^{\cocharge(C)}\, 
 x^{\sh(C)}
=
\sum_{colored~word~w\atop \weight(w)=\mu} 
%\sum_{\beta: \des(w)\subseteq{\rm set}(\beta)}
\sum_{\beta: \sh(w)=\beta}
q^{\ski(w)}\, t^{\cocharge(w)}\, x^\beta
\,.
$$
The claim follows by noting that
$$
\sum_{\beta: \sh(w)=\beta} x^\beta
=
\sum_{\substack{\gamma\models\ell(w)\\ n-\des(w)\subseteq{\rm set}(\gamma)}} 
\sum_{i_1<\cdots<i_{\ell(\gamma)}}
x_{i_1}^{\gamma_1}\cdots x_{i_{\ell(\gamma)}}^{\gamma_{\ell(\gamma)}}\,.
$$
\end{proof}

\bibliographystyle{alpha}
\bibliography{circloids}

\newcommand{\etalchar}[1]{$^{#1}$}
\begin{thebibliography}{HHL{\etalchar{+}}05b}

\bibitem[Bry89]{Bry}
R.~Brylinski.
\newblock Limits of weight spaces, lusztig's q-analogs, and fiberings of
  adjoint orbits.
\newblock {\em J. Amer. Math. Soc.}, 2(3):517--533, 1989.

\bibitem[Buc02]{Buch}
A.~Buch.
\newblock {A Littlewood-Richardson rule for the $K$-theory of Grassmannians}.
\newblock {\em Acta Math}, 189:36--78, 2002.

\bibitem[Che95]{Cherednik}
I.~Cherednik.
\newblock {Double affine Hecke algebras and Macdonald's conjectures}.
\newblock {\em Ann. of Math}, 141(1):191--216, 1995.

\bibitem[EAK94]{EK}
P.~Etingof and Jr. A.~Krillov.
\newblock {Macdonald's polynomials and representations of quantum groups}.
\newblock {\em Math. Res. Lett.}, 1(3):279--296, 1994.

\bibitem[For92]{For}
P.~J. Forrester.
\newblock {Selberg correlation integrals and the} $1/r^2$ {quantum many-body
  system}.
\newblock {\em Nucl. Phys. B}, 388(3):671--699, 1992.

\bibitem[GH93]{GHqt}
A.~M. Garsia and M.~Haiman.
\newblock {A graded representation model for Macdonald's polynomials}.
\newblock {\em Proc. Natl. Acad.}, 90(8):3607--3610, 1993.

\bibitem[GP92]{GP}
A.~M. Garsia and C.~Procesi.
\newblock {On certain graded $S_n$-modules and the $q$-Kostka polynomials}.
\newblock {\em Adv. Math.}, 87:82--138, 1992.

\bibitem[Gre55]{Green}
J.~A. Green.
\newblock The characters of the finite general linear groups.
\newblock {\em Trans. Amer. Math. Soc.}, 80:402--447, 1955.

\bibitem[Hai01]{Haiman}
M.~Haiman.
\newblock {Hilbert schemes, polygraphs, and the Macdonald positivity
  conjecture}.
\newblock {\em J. Amer. Math. Soc.}, 14:941--1006, 2001.

\bibitem[HHL05a]{HHL}
J.~Haglund, M.~Haiman, and N.~Loehr.
\newblock {A combinatorial formula for Macdonald polynomials}.
\newblock {\em J. Amer. Math. Soc.}, 18(3):735--761, 2005.

\bibitem[HHL{\etalchar{+}}05b]{HHLRU}
J.~Haglund, M.~Haiman, N.~Loehr, J.~B. Remmel, and A.~Ulyanov.
\newblock {A combinatorial formula for the character of the diagonal
  coinvariants}.
\newblock {\em Duke Math}, 126(2):195--232, 2005.

\bibitem[HS77]{HoSpr}
R.~Hotta and T.~A. Springer.
\newblock {A specialization theorem for certain Weyl group representations and
  an application to the Green polynomials of unitary groups}.
\newblock {\em Invent. Math.}, 41(2):113--127, 1977.

\bibitem[Kas90]{Kas1}
M.~Kashiwara.
\newblock {Crystalizing the q-analogue of universal enveloping algebras}.
\newblock {\em Comm. Math. Phys.}, 133(2):249--260, 1990.

\bibitem[Kas91]{Kas2}
M.~Kashiwara.
\newblock {On crystal bases of the q-analogue of universal enveloping
  algebras}.
\newblock {\em Duke Math. J.}, 63(2):465--516, 1991.

\bibitem[Las03]{Ldouble}
A.~Lascoux.
\newblock {\em {Double crystal graphs in Studies in Memory of Issai Schur}},
  volume 210.
\newblock Progress in Mathematics, 2003.

\bibitem[LLM03]{[LLM]}
L.~Lapointe, A.~Lascoux, and J.~Morse.
\newblock {Tableau atoms and a new Macdonald positivity conjecture}.
\newblock {\em Duke Math. J.}, 116(1):103--146, 2003.

\bibitem[LLT02]{LLTplactic}
A.~Lascoux, B.~Leclerc, and J.~Y. Thibon.
\newblock {The plactic monoid in Algebraic Combinatorics on Words, by M.
  Lothaire, ed.}
\newblock {\em Encyclopedia of Math. and Appl., Cambridge University Press},
  90:164--196, 2002.

\bibitem[LMS17]{LMS}
H.~Li, J.~Morse, and P.~Shields.
\newblock {Structure constants for $K$-theory of Grassmannians, revisited}.
\newblock {\em J. of Comb. Theory, Ser. A}, 144:306--325, 2017.

\bibitem[LS78]{[LS2]}
A.~Lascoux and M.-P. Sch\"utzenberger.
\newblock {Sur une conjecture de H.O. Foulkes}.
\newblock {\em C.R. Acad. Sc. Paris}, 294:323--324, 1978.

\bibitem[LS81]{LSplactic}
A.~Lascoux and M.-P. Sch\"utzenberger.
\newblock {Le mono\" \i de plaxique}.
\newblock {\em Quaderni della Ricerca Scientifica}, 109:129--156, 1981.

\bibitem[LS13]{LenSch}
C.~Lenart and A.~Schilling.
\newblock Crystal energy functions via the charge in types {A} and {C}.
\newblock {\em Math Z.}, 273:401--426, 2013.

\bibitem[Lus81]{Lus}
G.~Lusztig.
\newblock Green polynomials and singularities of unipotent classes.
\newblock {\em Adv. in Math}, 42:169--178, 1981.

\bibitem[Lus83]{[Lu]}
G.~Lusztig.
\newblock {Some examples of square integrable representations of semisimple
  $p$-adic group}.
\newblock {\em Trans. Amer. Math. Soc.}, 277:623--653, 1983.

\bibitem[Mac95]{Macbook}
I.~G. Macdonald.
\newblock {\em {Symmetric functions and Hall polynomials}}.
\newblock Clarendon Press, Oxford, 2nd edition, 1995.

\bibitem[ORS12]{ors}
A.~Oblomkov, J.~Rasmussen, and V.~Shende.
\newblock {The Hilbert scheme of a plane curve singularity and the HOMFLY
  homology of its link}.
\newblock {\em Duke Math. J.}, 161(7):1277--1303, 2012.

\bibitem[PY17]{PY}
O.~Pechenik and A.~Yong.
\newblock {Equivariant K-theory of Grassmannians II: The Knutson-Vakil
  conjecture}.
\newblock {\em Compos. Math.}, 153:667--677, 2017.

\bibitem[Rob17]{Roberts}
Austin Roberts.
\newblock {On the Schur expansion of Hall-Littlewood and related polynomials
  via Yamanouchi words}.
\newblock {\em Elec. J. of Comb}, 24(1), 2017.

\bibitem[Sch77]{Schutz}
M.-P. Sch\"utzenberger.
\newblock {La correspondance de Robinson}.
\newblock {\em Foata, Dominique, Combinatoire et repr\'esentation du groupe
  sym\'etrique (Actes Table Ronde CNRS, Univ. Louis-Pasteur Strasbourg,
  Strasbourg, 1976)}, 579:59--113, 1977.

\bibitem[Shi]{Sdummy}
Mark Shimozono.
\newblock {Crystals for Dummies}.
\newblock {\em {Expository Notes}}.

\bibitem[Spr78]{Spr}
A.~Springer.
\newblock {A construction of representations of Weyl groups}.
\newblock {\em Invent. Math.}, 44(3):279--293, 1978.

\bibitem[SV11]{ScVa}
A.~Schiffmann and E.~Vasserot.
\newblock {The elliptic Hall algebra, Cherednick Hecke algebras and Macdonald
  polynomials}.
\newblock {\em Compos. Math.}, 147(1):188--234, 2011.

\bibitem[vL01]{VanL}
Marc~A.A. van Leeuwen.
\newblock {The Littlewood-Richardson rule, and related combinatorics}.
\newblock {\em Math. Soc. of Japan, Memoirs}, 11:95--145, 2001.

\end{thebibliography}
\label{sec:biblio}

\end{document}